\documentclass[a4paper,Haag duality]{mathscan}
\newtheorem{thm}{Theorem}[section] 
\newtheorem{pro}[thm]{Proposition}  
\newtheorem{cor}[thm]{Corollary}    
\newtheorem{lem}[thm]{Lemma}        
\theoremstyle{definition}           
\newtheorem{rem}[thm]{Remark}       

\newcommand{\NI}{\noindent}

\newcommand{\bea}{\begin{eqnarray}}
\newcommand{\eea}{\end{eqnarray}}
\newcommand{\dsp}{\displaystyle}
\def \b #1 {\bf #1}
\newcommand{\IR}{\mathbb{R}}

\newcommand{\IM}{\mathbb{M}}
\newcommand{\IE}{\mathbb{E}}
\newcommand{\IC}{\mathbb{C}}
\newcommand{\ID}{\mathbb{D}}

\newcommand{\IT}{\mathbb{T}}
\newcommand{\IZ}{\mathbb{Z}}
\newcommand{\IP}{\mathbb{P}}
\newcommand{\cal}{\mathcal}
\newcommand{\clk}{{\cal K}}

\newcommand{\cla}{{\cal A}}
\newcommand{\clm}{{\cal M}}

\newcommand{\cli}{{\cal I}}
\newcommand{\cls}{{\cal S}}

\newcommand{\clf}{{\cal F}}

\newcommand{\clh}{{\cal H}}
\newcommand{\clp}{{\cal P}}
\newcommand{\clo}{{\cal O}}
\newcommand{\clb}{{\cal B}}

\newcommand{\clj}{{\cal J}}
\newcommand{\cln}{{\cal N}}
\newcommand{\cld}{{\cal D}}

\newcommand{\clc}{{\cal C}}

\newcommand{\raro}{\rightarrow}

\newcommand{\vsp}{\vskip 1em}

\newcommand{\ul}{\underline}

\def \qed {\hfill \vrule height6pt width 6pt depth 0pt}
\newcommand{\be}{\begin{equation}}
\newcommand{\ee}{\end{equation}}
\newcommand{\ben}{\begin{eqnarray*}}
\newcommand{\een}{\end{eqnarray*}}
\begin{document}

\title{Isomorphism theorem for Kolmogorov states of $\IM=\dsp{\otimes_{n \in \IZ}}\!M^{(n)}_d(\IC)$}

\author{ Anilesh Mohari }

\address{ The Institute of Mathematical Sciences, CIT Campus, Taramani, Chennai-600113 }

\email{anilesh@imsc.res.in}

\keywords{Uniformly hyperfinite factors, Kolmogorov's property, Mackey's imprimitivity system, CAR algebra, Quantum Spin Chain, Simple $C^*$ algebra, Tomita-Takesaki theory, norm one projection }

\subjclass{46L}

\thanks{ The author gratefully acknowledge and thanks the reviewers for their constructive comments which helped him to include the finer details of the proof in order to make 
the present paper more accessible. }

\begin{abstract}
We consider the translation dynamics on the $C^*$-algebra $\IM =\otimes_{n \in \IZ}\!M^{(n)}(\IC)$ of two sided infinite tensor product of $d$ dimensional matrices 
$\!M^{(n)}(\IC)=\!M_d(\IC)$ over the field of complex numbers $\IC$ and its restriction
to the maximal abelian $C^*$ sub-algebra $\ID^e = \otimes_{n \in \IZ }\!D_e^{(n)}(\IC)$ of
$\IM$, where each $\!D_e^{(n)}(\IC)=\!D_d(\IC)$ is the algebra of $d$ dimensional diagonal matrices with respect to an orthonormal basis $e=(e_i)$ of $\IC^d$. We prove that any two translation invariant Kolmogorov pure states of $\IM$ give unitarily equivalent dynamics in their Gelfand-Naimark-Segal spaces. Furthermore, for a class of 
Kolmogorov pure states of $\IM$ satisfying some additional invariant, we prove Kolmogorov states give isomorphic translation dynamics if their restrictions to the maximal abelian $C^*$ sub-algebra $\ID^e$ of $\IM$ are isomorphic. 

\vsp 
On the other hand, we prove that the translation dynamics with two infinite tensor product translation invariant faithful states of $\IM$ are isomorphic if and only if their mean entropies are equal.           
\end{abstract}

\maketitle 

\section{ Introduction }

\vsp 
A stationary finite Markov chain gives a translation invariant Markov state on a two sided classical spin chain. The set of translation invariant Markov states is a closed subset in the set of all translation invariant states on the two sided classical spin chain. One celebrated result in ergodic theory says that two such translation invariant states with positive Kolmogorov-Sinai dynamical [CFS] entropies are isomorphic [Or70] if and only if their dynamical entropies are equal. This classification program was primarily formulated [CFS] with a motivation to classify the symbolic dynamics [Si] associated with automorphic systems of classical Hamiltonian dynamics with Kolmogorov property. Though the set of ergodic states is dense in the set of two sided translation invariant states, all ergodic states need not be Markov states. In other words, all symbolic dynamics are not associated with stationary Markov chains [Or73] and a complete classification of automorphic systems remains incomplete in the most general mathematical set up of classical dynamical systems. Nevertheless these results found profound use in classical ergodic theory in special situations of paramount importance and remains an active area of research over the last few decades due to its diverse applications to other areas of mathematics [CFS].    

\vsp 
In this paper, we set a program aiming towards a classification of translation invariant states of the two sided quantum spin chain motivated to understand properties of translation invariant factor and pure states that appear naturally as temperature and ground states of quantum mechanical Hamiltonian systems [BRII,Sim]. Though our results are similar in spirit with results in classical ergodic theory, our technique and motivations are quite different. Before we specialize to two sided quantum spin chain, we recall some standard terminologies and notations used in operator algebras [Ta2]. We will also recall the classical situation of our present problem in some details in the following for a comparison with its quantum counter part. 

\vsp 
Let $\clb$ be a $C^*$-algebra over the field of complex numbers $\IC$. A linear functional $\omega:\clb \raro \IC$ is called a {\it state} on $\clb$ if it is {\it positive} i.e. $\omega(x^*x) \ge 0$ for all $x \in \clb$ and {\it unital} i.e. $\omega(I)=1$, where $I$ is the unit element of $\clb$. The convex set $\clb^*_{+,1}$ of states on $\clb$ is compact in weak$^*$ topology of the dual Banach space $\clb^*$ of $\clb$. A state $\omega$ of $\clb$ is called {\it pure} if the state can not be expressed as a convex combination of two different states i.e. if $\omega$ is an extremal element in $\clb^*_{+,1}$.  

\vsp 
Let $\clb_1,\clb_2$ be two unital $C^*$-algebras [BR,Ta2] over the field of complex numbers $\IC$. A unital linear map $\pi: \clb_1 \raro \clb_2$ is called {\it homomorphism} if 
\be 
\pi(x)^*=\pi(x^*)\;\;\mbox{and}\;\; \pi(xy)=\pi(x)\pi(y)
\ee 
for all $x,y \in \clb_1$. An injective homomorphism $\beta:\clb_1 \raro \clb_2$ is called 
{\it endomorphism}. For an unital $C^*$ algebra $\clb$, a linear bijective map $\theta: \clb \raro \clb$ is called {\it automorphism} if the map is a homomorphism. 

\vsp 
For an unital $C^*$ algebra $\clb$, a state $\omega$ of $\clb$ is called {\it invariant} 
for an automorphism $\theta: \clb \raro \clb$ if $\omega = \omega \theta$. A triplet $(\clb,\theta,\omega)$ is called a unital $C^*$-{\it dynamical system} if $\clb$ is a unital $C^*$-algebra and $\theta:\clb \raro \clb$ is an automorphism preserving a state $\omega$ of $\clb$. For a given automorphism $\theta$ on a unital $C^*$-algebra $\clb$, the set of invariant states 
$$\cls^{\theta}=\{ \omega \in \clb^*_{+,1}: \omega = \omega  \theta \}$$ 
of $\theta$ is a non empty compact set in the weak$^*$ topology of the dual Banach space $\clb^*$. An extremal element in the convex set $\cls^{\theta}$ is called {\it ergodic } state for $\theta$. An invariant state $\omega$ of $\theta$ is ergodic if and only if 
\be 
{1 \over 2n+1} \sum_{-n \le k \le n} \omega(y\theta^k(x)z) \raro \omega(yz)\omega(x)
\ee
as $n \raro \infty$ for all $x,y,z \in \clb$. 

\vsp
Let $(\clh_{\omega},\pi_{\omega},\zeta_{\omega})$ be the Gelfand-Naimark-Segal (GNS) space associated with a state $\omega$ of $\clb$, where $\pi_{\omega}:\clb \raro \clb(\clh_{\omega})$ is a $*$-representation of $\clb$ and $\zeta_{\omega}$ is the cyclic vector for $\pi(\cla)$ in $\clh_{\omega}$ such that 
$$\omega(x)= \langle \zeta_{\omega},\pi_{\omega}(x)\zeta_{\omega} \rangle$$
Let $\pi_{\omega}(\clb)'$ be the {\it commutant} of $\pi_{\omega}(\clb)$ i.e. $\pi_{\omega}(\clb)'=\{x \in \clb(\clh_{\omega}): xy=yx,\;\forall y \in \pi_{\omega}(\clb) \}$ and $\pi_{\omega}(\clb)''$ be the {\it double commutant} of $\pi_{\omega}(\clb)$ i.e.
$\pi_{\omega}(\clb)''= \{ x \in \clb(\clh_{\omega}): xy=yx\;\;\forall y \in \pi_{\omega}(\clb)' \}$. By a celebrated theorem of von-Neumann, $\pi_{\omega}(\clb)''$ is the weak$^*$ completion of $\pi_{\omega}(\clb)$ in $\clb(\clh_{\omega})$ and it admits a pre-dual Banach space. The pre-dual Banach space is often called {\it normal functional} on $\pi_{\omega}(\clb)''$. A state $\omega$ of $\clb$ is pure if and only if $\pi_{\omega}(\clb)''=\clb(\clh_{\omega})$, the algebra of all bounded operators on $\clh_{\omega}$.

\vsp 
Given a unital $C^*$-dynamical system $(\clb,\theta,\omega)$, we have a unitary operator $S_{\omega}:\clh_{\omega} \raro \clh_{\omega}$ extending the following inner product preserving map
\be 
S_{\omega} \pi_{\omega}(x)\zeta_{\omega}=\pi_{\omega}(\theta(x))\zeta_{\omega},\;x \in \clb
\ee 
and an automorphism $\Theta_{\omega}:\pi_{\omega}(\clb)'' \raro \pi_{\omega}(\clb)''$, defined by 
\be 
\Theta_{\omega}(X)=S_{\omega}XS^*_{\omega},\;X \in \pi_{\omega}(\clb)''
\ee 
Thus we have  
$$\Theta_{\omega}(\pi_{\omega}(x)) = \pi_{\omega}(\theta(x)),\;x \in \clb$$
Furthermore, $\omega$ is ergodic for $\theta$ if and only if 
$$\{ f: S_{\omega}f=f,\; f \in \clh_{\omega} \} = \{z \zeta_{\omega}: z \in \IC \}$$ 
Though there is no direct meaning in the non commutative framework for `an individual ergodic theorem' [Pa] of G. D. Birkhoff for automorphism $\Theta_{\omega}$ on the von-Neumann algebra $\pi_{\omega}(\clb)''$ with a normal invariant vector state given by the unit vector $\zeta_{\omega}$ on $\pi_{\omega}(\clb)''$, it has a non commutative generalization in a more general sense [La]. 

\vsp 
An invariant state $\omega$ for $\theta$ is called {\it strongly mixing} if
\be 
\omega(x\theta^n(y)) \raro \omega(x)\omega(y)
\ee
as $|n| \raro \infty $ for all $x,y \in \clb$. A strongly mixing state is obviously ergodic, however the converse is false. A simple application of Riemann-Lebesgue lemma says that absolute continuous spectrum of $S_{\omega}$ in the orthogonal complement of invariant vector in $\clh_{\omega}$ of $S_{\omega}$ is sufficient for strong mixing property [Pa]. It is not known yet, whether converse is true. In other words, no simple criteria on $\omega$ is known yet for strongly mixing. 

\vsp 
In case, $\clb$ is a unital commutative $C^*$-algebra, then $\clb$ is isomorphic to $C(X)$, where $C(X)$ is the algebra of complex valued continuous functions on a compact Hausdorff space $X$. An automorphism $\theta$ on $C(X)$ determines a unique one to one and onto continuous point map $\gamma_{\theta}:X \raro X$ such that  $\theta(f)=f \circ \gamma_{\theta}$, for all $f \in C(X)$. A state $\omega$ on $\clb \equiv C(X)$ is determined uniquely by a regular probability measure $\mu_{\omega}$ on $X$ by $\omega(f)=\int f d \mu_{\omega}$ and its associated GNS space $\clh_{\omega}=L^2(X,\mu_{\omega})$ with representation $\pi_{\omega}(h)f=hf$ for all $f \in C(X)$ with $\pi_{\omega}(X)''=L^{\infty}(X,\clf_X,\mu_{\omega})$, where $\clf_X$ is the Borel $\sigma$-field of $X$. Thus ergodic and strongly mixing properties introduced in the non commutative framework of $C^*$ algebras are in harmony with its classical counter parts and coincide once one restricts them to commutative $C^*$-algebras. 

\vsp 
Two unital $C^*$-dynamical systems $(\clb_1,\theta_1,\omega_1)$ and $(\clb_2,\theta_2,\omega_2)$ are said to be {\it isomorphic } if there exists a $C^*$ automorphism $\beta:\clb_1 \raro \clb_2$ such that 
\be 
\theta_2  \beta = \beta   \theta_1,\;\;\omega_2  \beta=\omega_1
\ee 
on $\clb_1$. It is clear that ergodic and strong mixing properties remain covariant with respect to the isomorphism. One of the central problem in classical ergodic theory is aimed to classify classical dynamical system of automorphisms upto the isomorphism. We replace automorphisms $(\theta_k:k=1,2)$ by unital endomorphisms in these definitions to include more general $C^*$-dynamical systems of $*$-endomorphisms. 

\vsp
A state $\omega$ of $\clb$ is called {\it factor} if $\pi_{\omega}(\clb)''$ is a factor i.e. if the centre $\pi_{\omega}(\clb)'' \bigcap \pi_{\omega}(\clb)'$ is equal to $\{ z \pi_{\omega}(I):z \in \IC \}$. An automorphism of $\clb$ takes a factor state to another factor state. Apart from ergodic and strong mixing properties, factor property is also an invariant under the isomorphism of two dynamics. For a commutative $C^*$-algebra $\clb$, the centre of $\pi_{\omega}(\clb)''$ is itself and thus $\omega$ is a factor state if and only if $\clh_{\omega}$ is one dimensional. In other words, $\omega$ is a Dirac measure on a point of $X$, where $\clb \equiv C(X)$. Such a state plays no interest in classical dynamical system since invariance property for the automorphism ensures that $\gamma_{\theta}$ has a fixed point in $X$.   

\vsp 
We introduce now one more invariant in the general set up of $C^*$-dynamical system. 
For a family of $C^*$ sub-algebras $(\clb_i:i \in \cli)$ of $\clb$, we use the 
notation $\vee_{i \in \cli} \clb_i $ for the $C^*$ algebra generated by the family 
$(\clb_i: i \in \cli)$. For a $C^*$ sub-algebra $\clb_0$ of $\clb$, we set $\clb'_0 
= \{ x \in \clb: xy=yx,\; \forall y \in \clb_0 \}$ and $\clb_0''= \{ x \in \clb: xy=yx,\; \forall y \in \clb_0' \}$. A $C^*$-dynamical system $(\clb,\theta,\omega)$ is said to have {\it Kolmogorov property} if there exists a $C^*$ sub-algebra $\clb_0$ of $\clb$ such that the following hold:

\vsp 
\NI (a) $\clb_0''=\clb_0,\;\;\theta^{-1}(\clb_0) \subseteq \clb_0$;

\vsp 
\NI (b) $\vee_{n \in \IZ} \theta^n(\clb_0) = \clb$; 

\vsp
\NI (c) $\bigcap_{n \in \IZ} \theta^{-n}(\clb_0)=\{z I: z \in \IC \}.$ 

\vsp 
\NI (d) For each $n \in \IZ$, let $F^{\omega}_{n]}$ be the projection $[\pi_{\omega}(\theta^n(\clb_0))\zeta_{\omega}]$ in $\clh_{\omega}$. Then  
$$F^{\omega}_{n]}\;\;\downarrow \;\;|\zeta_{\omega} \rangle \langle \zeta_{\omega}|$$
in strong operator topology as $n \raro -\infty$. In short, such an element $\omega \in \cls^{\theta}(\clb)$ is called {\it Kolmogorov state}. When a fixed state $\omega \in \cls^{\theta}(\clb)$ is under consideration, we omit the superscript $\omega$ 
in the notation $F^{\omega}_{n]}$ and denote by $F_{n]}$ for each $n \in \IZ$. 

\vsp 
The relations (a)-(c) are state independent. In particular, for any translation invariant state $\omega$ of $\clb$, properties (a) and (b) ensure that 
$$F_{n]} \uparrow I_{\clh_{\omega}}$$ as $n \uparrow \infty$. Furthermore, $F_{n]} \downarrow F_{-\infty]}$ as $n \downarrow -\infty$ in strong operator topology for some projection $F_{-\infty]} \ge |\zeta_{\omega} \rangle \langle \zeta_{\omega}|$. However, in general $F_{-\infty]}$ need not be equal to $|\zeta_{\omega} \rangle \langle \zeta_{\omega}|$ even when (c) is true. Thus the property (d) is crucial to determine Kolmogorov property of the state $\omega$. It is clear that Kolmogorov property is an invariant for the dynamics $(\clb,\theta,\omega)$. 

\vsp 
For a Kolmogorov state $\omega$ and any $x,y \in \clb_0$, we also have 
$$|\omega(x\theta^n(y)|$$
$$=|<x^*\zeta_{\omega}, F_{n]} \theta^n(y) \zeta_{\omega}>|$$
$$\le ||F_{n]}x^*\zeta_{\omega}||\;||\theta^n(y)\zeta_{\omega}||$$
$$\le ||F_{n]}x^*\zeta_{\omega} || \; ||y||$$ 
$$ \raro 0$$
as $n \raro -\infty$ once $\omega(x)=0$. Thus by linear property of the map $\theta$, strong mixing property (5) holds for all $x,y \in \clb_0$. Going along the same line of the proof, we also verify (5) for all $x,y \in \theta^{-m}(\clb_0),\; m \ge 1$. Since
$\cup_{m \ge 1} \theta^{-m}(\clb_0)$ is norm dense in $\clb$, a standard density argument gives (5) for all $x,y \in \clb$. In other words, Kolmogorov states are 
strongly mixing. By the argument used above, for each $x \in \clb$ we have
\be 
\mbox{sup}_{\{y:||y|| \le 1 \}} |\omega(x \theta^n(y))-\omega(x)\omega(y)| \raro 0
\ee
as $n \raro -\infty$.       

\vsp 
The $C^*$ dynamical systems $(\clb,\theta,\omega)$ is said to have {\it backward Kolmogorov property } if $(\clb,\theta^{-1},\omega)$ admits Kolmogorov property.  
Clearly the backward Kolmogorov property is an invariant as well. A theorem of 
Rokhlin-Sinai [Pa] says that the Kolmogorov property for classical dynamical systems 
is equivalent to strictly positive Kolmogorov-Sinai dynamical entropy. Since Kolmogorov-Sinai dynamical entropies are equal for $(\clb,\theta,\omega)$ and $(\clb,\theta^{-1},\omega)$, Kolmogorov property is time-reversible. However, in non commutative framework, such a time reversal property [Mo2] for a Kolmogorov state is not clear even though relations (a)-(c) hold for $(\clb,\theta^{-1},\omega)$ with $\theta$ and $\clb_0'$ replacing $\theta^{-1}$ and $\clb_0$ respectively.  

\vsp 
A vector subspace $\cli$ of a $C^*$-algebra $\clb$ is called {\it ideal} or {\it two sided ideal } of $\clb$ if 

\NI (a) $\cli$ is closed under conjugation i.e. $x^* \in \cli$ if $x \in \cli$;

\NI (b) $xy,yx \in \cli$ for all $x \in \cli$ and $y \in \clb$. 

\vsp 
A $C^*$ algebra $\clb$ is called {\it simple} if $\clb$ has no proper ideal i.e. other then $\clb$ or $\{0\}$. 

\vsp 
For a simple $C^*$ algebra $\clb$, any non trivial $*$-homomorphism $\beta:\clb \raro \clb$ is injective since the null space $\cln=\{x \in \clb_1: \beta(x)=0 \}$ is a two sided ideal. Thus $x \raro ||\beta(x)||$ is a $C^*$ norm on $\clb$. Since $C^*$ norm is unique on a $C^*$ algebra with a given involution by Gelfand spectral theorem [BR1], we get 
$$||\beta(x)||=||x||$$ 
for all $x \in \clb$. In particular, a homomorphism $\beta:\clb \raro \clb$ is an automorphism for a simple $C^*$ algebra $\clb$ if the homomorphism $\beta$ is onto.     

\vsp 
In the following, we will investigate this abstract notion of Kolmogorov property in more details in the context of one lattice dimension two sided quantum spin chains studied in a series of papers [Mo1],[Mo2] and [Mo4]. We also find its relation with the well known Kolmogorov property in classical dynamical systems [Pa].

\vsp 
In the following, we describe $C^*$ algebraic set up valid for quantum spin chain 
[BR vol-II,Ru] and find its relation to classical spin chain [Pa] in details. Let $\IM=\otimes_{n \in \IZ} \!M^{(n)}_d(\IC)$ be the $C^*$ -completion of the infinite tensor product of the algebra $\!M_d(\IC)$ of $d$ by $d$ matrices over the field of complex numbers. Let $Q$ be a matrix in $\!M_d(\IC)$. By $Q^{(n)}$ we denote the element $...\otimes I_d \otimes I_d \otimes I_d \otimes Q \otimes I_d \otimes I_d \otimes I_d \otimes ... $, where $Q$ appears in the $n$-th component in the tensor product and $I_d$ is the identity matrix of $\!M_d(\IC)$. Given a subset $\Lambda$ of $\IZ$, $\IM_{\Lambda}$ is defined to be the $C^*$-sub-algebra of $\IM$ generated by elements $Q^{(n)}$ with $Q \in \!M_d(\IC)$, $n \in \Lambda$. The $C^*$ $\IM$ being the inductive limit of increasing matrix algebras, it is a simple $C^*$-algebra [ChE],[SS]. 

\vsp 
We also set $$\IM_{loc}= \bigcup_{\Lambda:|\Lambda| < \infty } \IM_{\Lambda},$$
where $|\Lambda|$ is the cardinality of $\Lambda$. An automorphism $\beta$ on $\IM$ is called {\it local } if $\beta(\IM_{loc}) \subseteq \IM_{loc}$. Right translation $\theta$ is a local automorphism of $\IM$ defined by $\theta(Q^{(n)})=Q^{(n+1)}$. We also simplified notations $\IM_{R}=\IM_{[1,\infty)}$ and $\IM_{L}=\IM_{(-\infty,0]}$. Thus restriction of $\theta$ ( $\theta^{-1}$ ), $\theta_R$ (and $\theta_L$) is a unital $*$-endomorphisms on $\IM_R$ ($\IM_L$). 

\vsp 
We say a state $\omega$ of $\IM$ is {\it translation invariant} if $\omega  \theta = \omega$ on $\IM$. The restriction of $\omega$ to $\IM_{\Lambda}$ is denoted by $\omega_{\Lambda}$. We also set $\omega_{R}=\omega_{[1,\infty)}$ and $\omega_{L}=\omega_{(-\infty,0]}$. In such a case $(\IM_R,\theta_R,\omega_{R})$ and $(\IM_L,\theta_L,\omega_{L})$ are two unital dynamical systems of $*$-endomorphisms. 
In this paper, we are interested to deal with $C^*$ dynamical systems $(\IM,\theta,\omega)$ and its restriction to sub-algebras $(\IM_R,\theta_R,\omega_R)$. If $(\IM_R,\theta_R,\omega_R)$ and $(\IM_R,\theta_R,\omega'_R)$ are isomorphic with an intertwining automorphism $\beta_R:\IM_R \raro \IM_R$, i.e. $\theta_R  \beta_R = \beta_R  \theta_R$ and $\omega'_R=\omega_R  \beta$, then $(\IM,\theta,\omega)$ and $(\IM,\theta,\omega')$ are isomorphic with intertwining automorphism $\beta:\IM \raro \IM$, where $\beta$ is the inductive limit automorphism of $\beta_R:\IM_R \raro \IM_R$ determined by the universal property of inductive limit of $C^*$ algebras [Sak]. However the converse question is more delicate even when $\omega$ and $\omega'$ are two pure states of $\IM$ since their restrictions to $\IM_R$ need not be isomorphic [Mo4]. 

\vsp 
We begin with a simple technical result proved at the end of section 2. 

\vsp 
\begin{pro} 
Let $\IM_0$ be a $C^*$ sub-algebra of $\IM$ such that 

\NI (a) $\IM_0''=\IM_0,\;\theta^{-1}(\IM_0) \subseteq \IM_0$;

\NI (b) $\vee_{n \in \IZ} \theta^n(\IM_0) = \IM$;

\NI (c) $\bigcap_{n \in \IZ} \theta^n(\IM_0)=\IC$.

Then there exists an automorphism $\alpha$ on $\IM$ commuting with $\theta$ with $\alpha(\IM_L)= \IM_0$. 
\end{pro}

\vsp 
For a translation invariant state $\omega$ on $\IM$ and the GNS space $(\clh_{\omega},\pi_{\omega},\zeta_{\omega})$ of $(\IM,\omega)$, we set a sequence 
of increasing projections $F_{n]},\;n \in \IZ$ defined by 
\be 
F_{n]} = [\pi_{\omega}(\theta^n(\IM_L))''\zeta_{\omega}]
\ee 
and 
$$F_{-\infty]} = \mbox{lim}_{n \raro -\infty}F_{n]}$$ 

\vsp 
For a given $C^*$-dynamical system $(\IM,\theta,\omega')$ with Kolmogorov property, Proposition 1.1 says that there exists an isomorphic $C^*$-dynamical system 
$(\IM,\theta,\omega)$ with Kolmogorov property, where $\omega = \omega' \alpha$ 
for an automorphism $\alpha$ on $\IM$ commuting with $\theta$ and 
$$F_{-\infty]} = |\zeta_{\omega}\rangle\langle \zeta_{\omega}|$$
Furthermore, it also shows that the Kolmogorov property for the equivalent class is independent of the choice we make for $C^*$ sub-algebra $\IM_0$ of $\IM$.  
Thus without loss of generality while dealing with equivalence class of $C^*$ -dynamical system $(\IM,\theta,\omega)$, we say now onwards that $(\IM,\theta,\omega)$
is Kolmogorov if $F_{-\infty]} = |\zeta_{\omega} \rangle\langle \zeta_{\omega}|$. The argument that we have 
used to prove strongly mixing property of Kolmogorov states, as well gives a proof for factor property of Kolmogorov states by Theorem 2.5 in [Pow]. 
  
\vsp 
Simplest example of a Kolmogorov state is any infinite tensor product state ( in particular, the unique normalized trace ) $\omega = \otimes_{n \in \IZ} \omega_n$, 
where $\omega_n=\omega_{n+1}$ for all $n \in \IZ$. In such a case, we have   
\be 
F_{n]}\pi_{\omega}(x)F_{n]}=\omega(x)F_{n]}
\ee 
for all $x \in \IM_{[n+1,\infty)}$. This clearly shows $\omega$ is Kolmogorov. In particular, any two such tensor product states need not give isomorphic dynamics since the class of states includes both pure and non pure factor states. Thus for isomorphism problem, we need to deal with the class of Kolmogorov states with some additional invariants. In this paper, our analysis aims to deal with isomorphism problem for the classes of translation invariant Kolmogorov states with additional properties.

\vsp 
We also set a family of increasing projections $E_{n]}=[\pi_{\omega}(\theta^n(\IM_R))'\zeta_{\omega}],\; n \in \IZ$ i.e. $E_{n]}$ is the support projection of $\omega$ in $\pi_{\omega}(\theta^n(\IM_R))''$. Thus we have $E_{n]} \le E_{n+1]}$ and $E_{n]} \uparrow I_{\clh_{\omega}}$ as $n \uparrow \infty$. Let $E_{n]} \downarrow E_{-\infty]}$ as $n \downarrow -\infty$ for some projection $E_{-\infty]} \ge |\zeta_{\omega}\rangle\langle\zeta_{\omega}|$. It is clear that 
\be 
F_{n]} \le E_{n]}
\ee 
for all $n \in \IZ$. The inequality (10) is strict if $\omega$ is a non pure factor state [Mo5]. 

\vsp 
By a theorem in [Mo5], a translation invariant pure state $\omega$ of $\IM$ admits {\it Haag duality} property i.e. 
\be 
\pi_{\omega}(\IM_L)''= \pi_{\omega}(\IM_R)'
\ee
Though $\IM_R'=\IM_L$ as $C^*$ sub-algebras of $\IM$, it is a non trivial fact that the equality (11) holds for a translation invariant factor state $\omega$ if and only if the state $\omega$ is pure. In such a case, for each $n \in \IZ$, we have $F_{n]} = E_{n]}$. Thus a translation invariant pure state $\omega$ admits Kolmogorov property if and only if 
$E_{-\infty]}=|\zeta_{\omega}\rangle\langle\zeta_{\omega}|$. Thus our definition of Kolmogorov property for a translation invariant state is an extension of Kolmogorov property for a translation invariant pure state studied in [Mo4]. However, 
for a pure $\omega$, though $F_{n]}=E_{n]}$ for all $n \in \IZ$ [Mo5], the projection $F_{-\infty]}$ may not be equal to $|\zeta_{\omega}\rangle\langle\zeta_{\omega}|$. In the appendix, we have included an example of a pure state that fails to be Kolmogorov.  
Theorem 2.6 in [Mo4] says that the state $\omega$ is pure Kolmogorov if $\omega_R$ ($\omega_L$) is a type-I factor state ( $\pi_{\omega}(\IM_R)''$ is isomorphic to the algebra of bounded operators on a Hilbert space. At this stage, it is not clear whether the converse statement is also true for a pure Kolmogorov state.    

\vsp 
Given a translation invariant state $\omega$ of $\IM$, one has two important numbers: 

\vsp 
\NI (a) Mean entropy $s(\omega)=\mbox{limit}_{\Lambda_n \uparrow \IZ}{1 \over |\Lambda_n|}S_{\omega_{\Lambda_n}}$, where $S_{\omega_{\Lambda_n}}=-tr_{\Lambda}(\rho^{\omega}_{\Lambda}ln \rho^{\omega}_{\Lambda})$ is the von-Neumann entropy 
of the state $\omega_{\Lambda_n}(x)=tr_{\Lambda}(x\rho^{\omega}_{\Lambda})$ and 
$tr_{\Lambda}$ is the trace on $\IM_{\Lambda}$ i.e. von Neumann entropy of the restricted state $\omega$ to local $C^*$ algebra $\IM_{\Lambda_n}$, where $\Lambda_n=\{-n \le k \le n \}$ or more generally a sequence of finite subsets $\Lambda_n$ of $\IZ$ such that $\Lambda_n \uparrow \IZ$ in the sense of Van Hove [Section 6.2.4 in BR2]. It is not known yet whether $s(\omega)$ is an invariant 
for the translation dynamics $(\IM,\theta,\omega)$ i.e. whether $s(\omega)$ can 
be realized intrinsically as a dynamical entropy of $(\IM,\theta,\omega)$ for all translation invariant state $\omega$.    

\vsp 
\NI (b) Connes-St\o rmer dynamical entropy: $h_{CS}(\omega)$ [CS,CNT,OP,St\o2, NS] which is a close candidate for such an invariant for the translation dynamics $(\IM,\theta,\omega)$. It is known that $0 \le h_{CS}(\omega) \le s(\omega)$. In case $\omega$ is a product state then it is known that $h_{CS}(\omega)=s(\omega)$. It is also known that $h_{CS}(\omega)=0$ if $\omega$ is pure. However, no translation invariant state $\omega$ of $\IM$ is known in the literature for which $h_{CS}(\omega) < s(\omega)$. This result in particular shows that there is no automorphism $\beta$ on $\IM$ such that $\beta \theta = \theta^2 \beta$ [CS]. At this point, we warn attentive readers that a different notion of Kolmogorov dynamics is also studied in [NeS].  

\vsp 
At this stage, we may raise few non trivial questions: 

\vsp 
\NI (a) Does the mean entropy give a complete characterisation for the class of infinite tensor product states? In other words, does the equality in mean entropies of two infinite tensor product states give isomorphic translation dynamics of $\IM$? 

\vsp 
If answer for (a) is yes, then we may raise a more general question on isomorphism problem 
for translation invariant factor states of $\IM$: 

\vsp 
\NI (b) Does the mean entropy give a complete characterisation for the class of translation invariant factor states? In other words, given a factor state $\omega \in \cls_{\theta}(\IM)$, do we have an automorphism $\alpha$ on $\IM$ commuting with $\theta$ such that the translation invariant state $\omega \circ \alpha$ is an infinite tensor product state of $\IM$?  

\vsp 
Affirmative answers to both (a) and (b) in particular says that the mean entropy is a complete invariant for the class of factor states. In this paper we include a positive answer to question (a) in Theorem 5.1 which proves that the mean entropy is a complete invariant of translation dynamics for the class of infinite tensor product faithful states.  We also address the problem (b) under some additional restrictions on the class of factor states.  

\vsp 
Let $\Omega=\{1,2,.,d \}$, $\Omega^{\IZ}=\times_{n \in \IZ} \Omega^{(n)}$ be the infinite product set with $\Omega^{(n)}=\Omega$ for each $n \in \IZ$ and $C(\Omega^{\IZ})$ be the algebra of continuous complex valued functions on $\Omega^{\IZ}$. The commutative $C^*$-algebra $C(\Omega^{\IZ})$ is viewed as a $C^*$-sub-algebra of $\IM$ identified with the 
infinite tensor product $\ID^e = \otimes_{n \in \IZ} D_e^{(n)}(\IC)$ of diagonal matrices 
$D_e^{(n)}(\IC) \equiv D_d(\IC)$ of $\!M_d(\IC)$ with respect to an orthonormal basis 
$(e_i)$ for $\IC^d$. For a subset $\Lambda$ of $\IZ$, let $\Omega^{\Lambda} \subseteq \Omega^{\IZ}$ be the cylinder sets supported on $\Lambda$ and $\ID^e_{\Lambda}$ be the 
$C^*$ sub-algebra of $\ID^e$ spanned by the diagonal elements $D^{(n)}_e,\;n \in \Lambda$ 
in $\!M^{(n)}_d(\IC)$ with respect to the orthonormal basis $e=(e_i)$ of $\IC^d$. Thus the isomorphism $C(\Omega^{\IZ}) \equiv \ID^e$ also identifies $C(\Omega^{\Lambda}) \equiv \ID^e_{\Lambda}$.

\vsp 
Let $\cls^{\theta}(\ID^e)$ be the compact convex set of translation invariant states on $\ID^e$. Given an element $\omega_c \in \cls^{\theta}(\ID^e)$, we consider the non empty convex compact set 
\be 
\cls^{\theta}_{\omega_c}(\IM)= \{ \omega \in \cls^{\theta}(\IM),\; \omega_{|}\ID^e=\omega_c \}
\ee   
If $\omega_c$ is an extremal element in the set of translation invariant states of $\ID^e$, i.e. if $\omega_c$ is an ergodic state for $\theta:\ID^e \raro \ID^e$ then $\cls^{\theta}_{\omega_c}(\IM)$ is a convex face in the set of translation invariant state of $\IM$ and thus an extremal element in $\cls^{\theta}_{\omega_c}(\IM)$ is also ergodic. Question that arises at this point: how two extremal elements in $\cls^{\theta}_{\omega_c}(\IM)$ are related? 

\vsp 
In particular, for a probability measure $\mu$ on $\Omega$, we consider 
the Bernoulli measure $\mu_{\infty} = \otimes_{n \in \IZ} \mu^{(n)}$ on 
$C(\Omega^{\IZ})$, where each $\mu^{(n)}=\mu$. The state $\omega_{\mu_{\infty}}(f)=\int f d\mu_{\infty},\; f \in C(\Omega^{\IZ})$ has several Hann-Banach extensions to $\IM$ 
as translation invariant states of $\IM$. As an explicit example, let $\rho$ be a state on $\!M_d(\IC)$ so that $\rho(|e_j\rangle\langle e_j|)=\mu_j$ and $\omega_{\rho} = \otimes_{n \in \IZ} \rho^{(n)}$, where $\rho^{(n)}=\rho$ for all $n \in \IZ$. One trivial possibility is to take $\rho(|e_i\rangle\langle e_j|)=\delta^i_j \mu_j$. 
In such a case $\omega_{\rho}$ is not a pure state of $\!M_d(\IC)$ unless $\mu_j=1$ for some $j \in \Omega$ and associated product state $\omega_{\rho}$ on $\IM$ includes in particular, the unique tracial state on $\IM$ for a suitable choices of $\mu_j$ namely $\mu_j={1 \over d}$ for all $1 \le j \le d$. Furthermore for $d=2$ if $\rho(|e_i\rangle\langle e_j) = \delta^i_j \mu_j$ for some $0 < \mu_j < {1 \over 2}$, then we 
have Pukansky's type-III$_{\lambda}$ factor state $\omega_{\rho} = \otimes_{n \in \IZ} \rho^{(k)}$ (where $\rho^{(n)}=\rho$ for each $n \in \IZ$) for $\lambda=2 \mu_1 \in (0,1)$ [Pow]. More general we can ensure extensions of the Bernoulli measure $\mu_{\infty}$ to a product state of $\IM$ so that its mean entropy takes value of a prescribed number between $0$ and $\sum_j - \mu_jln(\mu_j)$. 

\vsp 
On the other hand we can choose a unit vector $\lambda=(\lambda_j)$ in $\IC^d$ and set pure state $\rho_{\lambda}(|e_i\rangle\langle e_j|)= \bar{\lambda_i} \lambda_j$, where $\mu_j=|\lambda_j|^2$. Associated tensor product state $\omega_{\lambda}=\otimes \rho^{(n)}_{\lambda}$ with each $\rho^{(n)}_{\lambda}=\rho_{\lambda}$ is a pure state 
of $\IM$. A translation invariant state $\omega$ of $\IM$ is said to be {\it pure Bernoulli } if $\omega$ is a pure product state of $\IM$. Such a pure Bernoulli state of $\IM$ is equal to $\omega_{\lambda}$ for some unit vector $\lambda \in \IC$. It is a simple observation that two such pure product states of $\IM$ give isomorphic translation dynamics. 

\vsp 
Thus the set of extremal elements in $\cls^{\theta}_{\mu_{\infty}}(\IM)$ contains both pure and non pure product states. Thus two such extremal elements in $\cls^{\theta}_{\mu_{\infty}}(\IM)$ are in general not related by an automorphism. The question that arises now: does equality in mean entropies of two such extremal elements give isomorphic translation dynamics of $\IM$?  

\vsp
Along the same line, instead of fixing a Bernoulli measure on $\Omega^{\IZ}$, we may fix a more general classical stationary Markov state $\omega_{\mu,p}$ on $\Omega^{\IZ}$ as follows. Let $p:\Omega \times \Omega \raro [0,1]$ be a transition probability matrix which admits a stationary probability measure $\mu:\Omega \raro [0,1]$ i.e. 
$$p^i_j \ge 0,\;\sum_j p^i_j=1$$ 
and 
$$\sum_j \mu_j = 1,\; \sum_i \mu_ip^i_j = \mu_j$$ 
Let $\{X_n:n \in \IZ\}$ be the two sided stationary Markov process or chain defined on a probability space $(\Omega^{\IZ},\clf,\omega_{\mu,p})$, $\clf$ be the $\sigma-$fields generated by the cylinder subsets of $\Omega^{\IZ}$ and $X_n(\zeta)=\zeta_n$ for $\zeta=(\zeta_n) \in \Omega^{\IZ}$ with 
\be 
\omega_{\mu,p}\{\zeta; X_{n+1}(\zeta)=j|X_n(\zeta)=i \} = p^i_j
\ee 
Let $\clf_{n]}$ be the $\sigma-$field generated by $X_k$ upto time $n$ i.e. $\clf_{n]}$ is the minimal $\sigma$-field that makes $X_k: -\infty < k \le n$ measurable. The process $(X_n)$ admits $(\clf_n)$-Markov property and $\omega_{\mu,p}$ is called {\it classical Markov state} on $\Omega^{\IZ}$. Such a Markov process $(X_n)$ is said to have {\it Kolmogorov property} if 
$$\bigcap_{n \in \IZ } \clf_{n]} = \{ \Omega^{\IZ},\emptyset \}$$ 
Such a Markov process is also called {\it Bernoulli shift } if $p^i_j=\mu_j$ for all $i,j \in \Omega$. It is well known [Pa,Or1,Or2] that a Bernoulli shift admits Kolmogorov property and a deeper result in ergodic theory says that the converse is also true i.e. a Markov chain with Kolmogorov property is isomorphic to a Bernoulli shift. A celebrated result in ergodic theory [Or1] says much more giving a complete invariant for such a class of Markov chain in terms of their Kolmogorov-Sinai dynamical entropy $h_{\mu}(\theta)$ [Pa] of the shift. An isomorphism that inter-twins the Kolmogorov shifts with equal positive KS dynamical entropy, need not inter-twin the filtrations generated by the Markov processes. However, an isomorphism between two such classical probability space is always local unlike an isomorphism on $\IM$. But the set of stationary Markov states is a closed proper subset of $\cls^{\theta}(\ID^e)$. Furthermore, all translation invariant states $\omega_c$ on $\ID^e$ are not isomorphic to Markov states $\omega_{\mu,p}$ [Or2].    

\vsp 
In this paper as a first step, we essentially deal with two simpler problems. Let $(\clh_{\omega_c},\pi_{\omega_c},\zeta_{\omega_c})$ be the GNS space associated with a translation invariant state $\omega_c$ of $\ID^e$ and $\clf_{n]}=[\pi_{\omega_c}(\ID^e_{\Lambda_n})\zeta_{\omega_c}]$ be the increasing projections in $\clh_{\omega_c}$, where $\Lambda_n=(-\infty,n],\; n \in \IZ$. For any fixed element $\omega \in \cls^{\theta}_{\omega_c}$, we have 
\be 
\clf_{n]} \le F_{n]} \le E_{n]},\;\; \forall n \in \IZ
\ee
We say $\omega_c$ is a {\it Kolmogorov} state of $\ID^e$ 
if 
$$
\mbox{lim}_{n \raro - \infty} \clf_{n]} = |\zeta_{\omega_c}\rangle\langle\zeta_{\omega_c}|
$$

\vsp 
Thus for a pure Bernoulli state $\omega$ of $\IM$, for each $n$ we have $F_{n]}=E_{n]}=|\zeta_{\omega}\rangle\langle\zeta_{\omega}|$ and thus it admits Kolmogorov property. Furthermore, note also that for a given pure Bernoulli state $\omega$ of $\IM$, the restriction $\omega_c$ of $\omega$ to $\!D^e$ is also Bernoulli in each orthonormal basis $e=(e_i)$ of $\IC^d$ with dynamical entropy with value in the range between $0$ and $ln(d)$. More generally, for any Kolmogorov state $\omega$ of $\IM$, the state $\omega_c = \omega_{|}\ID^e$ is also Kolmogorov by the inequality (14). 

\vsp 
Conversely, for a given Kolmogorov state $\omega_c$ of $\ID^e$, all its extreme points in $\cls^{\theta}_{\omega_c}(\IM)$ are not pure in general though factor states with unequal mean entropies. Nevertheless, this give rises the following valid questions for a Kolmogorov state $\omega_c$ of $\ID^e$: 

\vsp 
\NI (a) Is there an extreme point in the non empty convex compact set $\cls^{\theta}_{\omega_c}(\IM)$ with Kolmogorov property? 

\vsp 
\NI (b) Are all extreme points in $\cls^{\theta}_{\omega_c}(\IM)$ having Kolmogorov property?

\vsp 
\NI (c) Are two extreme points in $\cls^{\theta}_{\omega_c}(\IM)$ with pure Kolmogorov property related by an automorphism of $\IM$ commuting with $\theta$?

\vsp 
A Kolmogorov state of $\IM$ give rises a system of imprimitivity [Mac] for the group $\IZ$ as follows: Let $S:\clh_{\omega} \raro \clh_{\omega}$ be the unitary operator defined by extending 
$$S \pi_{\omega}(x)\zeta_{\omega}=\pi_{\omega}(\theta(x)\zeta_{\omega},\;x \in \IM$$ 
Then we have 
$$S^nF_{m]}(S^*)^n=F_{m+n]},\;m,n \in \IZ$$
and so 
$$S^nF'_{m]}(S^*)^n=F'_{m+n]},\;m,n \in \IZ,$$
where $F'_{n]}=F_{n]}-|\zeta_{\omega}\rangle\langle\zeta_{\omega}|,\;n \in \IZ$. 
Kolmogorov property ensures that $(-\infty,n] \raro F'_{n]}$ is a projection valued spectral measure on $\IZ$ and $(S^n,F_{n]}:n \in \IZ)$ once restricted to $|\zeta_{\omega}\rangle\langle\zeta_{\omega}|^{\perp}=I_{\clh_{\omega}}-|\zeta_{\omega}\rangle\langle\zeta_{\omega}|$ gives a system of imprimitivity for the group $\IZ$. By a theorem of G. W. Mackey [Mac], 
$I_{\clh_{\omega}} - |\zeta_{\omega}\rangle\langle\zeta_{\omega}|$ can be decomposed into copies of $l^2(\IZ)$ and the restriction of $(S^n)$ is isomorphic to direct sum of the standard shift on $l^2(\IZ)$. We call the number of copies that appear in the direct sum decomposition as {\it Mackey  index} for $\omega$. In the next section, we will prove Mackey index is always $\aleph_0$ for a  Kolmogorov state. Furthermore, we prove the following result.

\vsp 
\begin{thm} 
Any two pure Kolmogorov states of $\IM$ give unitarily equivalent translation dynamics i.e. there exists a unitary operator 
$U:\clh_{\omega_1} \raro \clh_{\omega_2}$ so that 
$$U\zeta_{\omega_1}=\zeta_{\omega_2},\;U\Theta_{\omega_1}(X)U^*=\Theta_{\omega_2}(UXU^*)$$ 
for all $X \in \pi_{\omega_1}(\IM_1)''$. 
\end{thm}

\vsp 
In the proof of Theorem 1.2, the simplicity of $\IM_0=\otimes_{n \ge 1}\!M^{(n)}(\IC)$ plays an important role and technique has a ready generalization to inductive limit states of injective endomorphism on a simple infinite dimensional separable $C^*$-algebra $\IM_0$. For a Kolmogorov state $\omega$, though the dynamics $(\IM,\theta,\omega)$ and $(\IM,\theta^2,\omega)$ are unitarily equivalent but they are not isomorphic. This shows that unitary equivalence is indeed weak, which fails to differentiate two dynamics $\theta$ and $\theta^2$. 

\vsp 
In contrast we have the following result on isomorphism for translation invariant pure state. 
 
\begin{thm} 
Let $\omega$ and $\omega'$ be two translation invariant pure states of $\IM$. Let  
$\zeta_{\omega}$ and $\zeta_{\omega'}$ be cyclic vectors for $\pi_{\omega}(\ID^e)''$ 
and $\pi_{\omega'}(\ID^e)''$ in $\clh_{\omega}$ and $\clh_{\omega'}$ respectively. If $(\ID^e,\theta,\omega_c)$ and $(\ID^e,\theta,\omega'_c)$ are isomorphic then $(\IM,\theta,\omega)$ and $(\IM,\theta,\omega')$ are isomorphic.     
\end{thm} 

\vsp 
In the following we reformulate the isomorphism problem of translation dynamics with a closely related problem in quantum dynamical systems with {\it completely positive maps } [St]  on a von-Neumann algebra and deal with isomorphism problem for a class of translation invariant pure states [FNW1,FNW2,BJKW] satisfying hypothesis of Theorem 1.3.    

\vsp 
Let $\clo_d$ be the simple C$^*$ algebra generated by $d$-many orthogonal isometries $(s_i,1 \le i \le d)$ [Cun] i.e. satisfying Cuntz relations
\be 
s^*_is_j=\delta^i_j I,\;\sum_{1 \le i \le d} s_is_i^*=I
\ee 
and 
\be 
\lambda(x)=\sum_is_ixs_i^*
\ee 
be the canonical endomorphism on $\clo_d$ which extends the right shift $\theta_R$ on $\IM_R=\otimes_{n \ge 1}M^{(n)}_d(\mathbb{C})$ and $\psi$ be a $\lambda$-invariant state i.e. $\psi \lambda=\psi$ on $\clo_d$. The universal property of Cuntz algebra gives a natural group action of unitary matrices $U_d(\mathbb{C})$ on $\clo_d$ given by 
\be 
\beta_g(s_i)=\sum_{1 \le j \le d} g^i_js_j,\;1 \le i \le d
\ee 
and $\lambda \beta_g= \beta_g \lambda$ for all $g \in U_d(\IC)$.  

\vsp 
We also consider the following non-empty convex weak$^*$ compact set
$$K_{\omega}=\{\psi\;\mbox{state of}\; \clo_d:\;\psi \lambda=\psi,\;\psi_{|}\mbox{UHF}_d=\omega_R \}$$ 
where we have identified $\mbox{UHF}_d$, the closed $C^*$ algebra generated by $\{s_Is_J^*:|I|=|J| <\infty\}$, with $\IM_R$ by fixing 
an orthonormal basis $(e_i)$ for $\IC^d$ which takes 
$$s_Is_J^* \raro |e_{i_1}\rangle\langle e_{j_1}|\otimes |e_{i_2}\rangle\langle e_{j_2}|\otimes ...\otimes |e_{j_n}\rangle\langle e_{j_n}|\otimes I_d \otimes ...$$ 
and 
$$\beta_g(s_Is_J^*) \raro ..\otimes g^* \otimes g^* (..|e_{i_1}\rangle\langle e_{j_1}|\otimes |e_{i_2}\rangle\langle e_{j_2}|\otimes ..\otimes |e_{j_n}\rangle\langle e_{j_n}|\otimes I_d \otimes ..)g \otimes g \otimes..$$

\vsp 
Let $(\clh_{\psi},\pi_{\psi},\zeta_{\psi})$ be the GNS space of $(\clo_d,\psi)$. We set support projection $P=[\pi_{\psi}(\clo_d)'\zeta_{\psi}]$ for the state $\psi$ in von-Neumann algebra $\pi_{\psi}(\clo_d)''$. 
Then by invariance property $\psi \lambda =\psi$, we have 
$\langle\zeta_{\psi},P\Lambda(I-P)P\zeta_{\psi}\rangle=0$ and thus $\Lambda(P) \ge P$, where $\Lambda(X)=
\sum_{1 \le i \le d} \pi_{\psi}(s_i) X \pi_{\psi}(s_i^*)$ for $X \in \pi_{\psi}(\clo_d)''$. 
We set unital completely positive map [St] $
\tau:\clm \raro \clm$ defined by 
\be
\tau(a)=P\Lambda(PaP)P
\ee   
for $a \in \clm=P\pi_{\psi}(\clo_d)''P$ with faithful normal invariant state $\phi(a)=\langle\zeta_{\psi},a\zeta_{\psi}\rangle$ for $a \in \clm$. 

\vsp 
It is known [BJKW,Mo5] that $\psi$ is a factor state of $\clo_d$ if and only if $(\clm,\tau,\phi)$ is ergodic i.e. no non-trivial invariant element in $\clm$ for the unital completely positive map $\tau$. Lemma 7.4 in [BJKW] proves that any extremal element $\psi$ in $K_{\omega}$ is a factor state if $\omega$ is a factor state and in such a case two extremal elements $\psi,\psi'$ of $K_{\omega}$ admits the relation 
$\psi'=\psi \beta_{zI_d}$ for some $z \in S^1=\{z \in \mathbb{C}: |z|=1 \}$. 
This shows 
that there are one to one relations between these ergodic $C^*$-dynamical systems modulo some gauge symmetry:
$$(C(\Omega^{\IZ}),\theta,\omega_c) \Rightarrow (\IM,\theta,\omega) \Leftarrow (\IM_R,\theta_R,\omega_R) \Leftrightarrow 
(\clo_d,\lambda,\psi) \Leftrightarrow (\clk,\clm,\tau,\phi)$$
and classification of any one of these dynamics are closely related with that of others. However the restriction of a translation invariant state $\omega$ to $\ID^e$ 
for a given orthonormal basis $e=(e_i)$, may not give a classical Markov chain [Na] 
unless embedded into a larger quantum spin chain.   

\vsp 
The paper is organized as follows. Section 2 gives a characterization of a local automorphism of $\IM$ which in particular gives a complete classification of one sided dynamical systems $(\IM_R,\theta_R,\omega_R)$ upto isomorphism and its relation classifying the dynamical system $(\clo_d,\lambda,\psi)$. Section 3 gives the proof 
of Theorem 1.2. Section 4 gives proof of our main result Theorem 1.3. Section 5 is essentially devoted to deal with a class of examples that fits as an application to Theorem 1.3. This section largely follows recent work of [DJ] on Markov state and its associated monic representation of Cuntz algebra. Section 6 includes a quick review of Tomita-Takesaki theory [Ta2, BR] and Takesaki's theorem on conditional expectation [Ta1]. As an application of our main tool developed in section 4, we prove that mean entropy is a complete isomorphism for infinite tensor product faithful state for translation dynamics of $\IM$. Section 7 is some what independent which includes a variation of a result [Pow], needed for the main theorem proved in section 4. The last section includes an appendix which gives a sketch of an example of a translation invariant pure state of $\IM$ which fails to have Kolmogorov property.

\section{Norm one projections} 

\vsp 
Let $\clm$ be a von-Neumann algebra acting on a Hilbert space $\clh$. A unit vector $\zeta$ is called cyclic for $\clm$ in $\clh$ if $[\clm \zeta]=\clh$. It is called separating for $\clm$ if $x\zeta=0$ for some $x \in \clm$ holds if and only if $x=0$.
Let $\clm'$ be the commutant of $\clm$, i.e. $\clm' = \{ x \in \clb(\clh): xy=yx \}$, where
$\clb(\clh)$ is the algebra of all bounded linear operators on $\clh$. An unit vector $\zeta$ is cyclic if and only if $\zeta$ is separating for $\clm'$.  

\vsp 
The closure of the closable operator $S_0:a\zeta \raro a^*\zeta,\;a \in \clm, S$ possesses a polar decomposition $S=\clj \Delta^{1/2}$, where $\clj$ is an anti-unitary and $\Delta$ is a non-negative self-adjoint operator on $\clh$. Tomita's [BR1] theorem says that 
\be 
\Delta^{it} \clm \Delta^{-it}=\clm,\;t \in \IR\;\mbox{and} \clj \clm \clj=\clm'
\ee 
We define the modular automorphism group
$\sigma=(\sigma_t,\;t \in \IT )$ on $\clm$
by
$$\sigma_t(a)=\Delta^{it}a\Delta^{-it}$$ which satisfies the modular relation
$$\omega(a\sigma_{-{i \over 2}}(b))=\omega(\sigma_{{i \over 2}}(b)a)$$
for any two analytic elements $a,b$ for the group of automorphisms $(\sigma_t)$. 
A more useful modular relation used frequently in this paper is given by 
\be 
\omega(\sigma_{-{i \over 2}}(a^*)^* \sigma_{-{i \over 2}}(b^*))=\omega(b^*a)
\ee 
which shows that $\clj a\zeta = \sigma_{-{i \over 2}}(a^*)\zeta$ for an analytic element $a$ for the automorphism group $(\sigma_t)$. The anti-unitary operator $\clj$ and the group of automorphism $\sigma=(\sigma_t,\;t \in \IR)$ are called {\it Tomita's conjugate operator} and {\it modular automorphisms } of the normal vector state $\omega_{\zeta}:x \raro \langle \zeta,x \zeta \rangle$ on $\clm$ respectively. 

\vsp 
A faithful normal state $\omega$ of $\clm$ is called stationary for a group $\alpha=(\alpha_t:\; t \in \IR)$ of automorphisms on $\clm$ if $\omega = \omega \alpha_t$ for all $t \in \IR$. A stationary state $\omega_{\beta}$ for $(\alpha_t)$ is called $\beta$-KMS state ($\beta > 0$) if there exists a function $z \raro f_{a,b}(z)$, analytic on the open strip $0 < Im(z) < \beta$, bounded continuous on the closed strip $0 \le Im(z) \le \beta$ with boundary condition 
\be 
f_{a,b}(t)=\omega_{\beta}(\alpha_t(a)b),\;\;f_{a,b}(t+i\beta)=\omega_{\beta}(\alpha_t(b)a)
\ee
for all $a,b \in \clm$. The faithful normal state $\omega_{\zeta}$ given by the cyclic and separating vector $\zeta$ is a $1 \over 2$-KMS state for the modular automorphisms group $\sigma=(\sigma_t)$. One celebrated theorem of M. Takesaki [Ta2] says that the converse statement is also true: If the normal state $\omega_{\zeta}$ given by cyclic and separating vector $\zeta$ is a $1 \over 2$-KMS state for a group $\alpha=(\alpha_t)$ of automorphisms on $\clm$ then $\alpha_t=\sigma_t$ for all $t \in \IR$.     

\vsp 
Let $\zeta_{\omega_1}$ and $\zeta_{\omega_2}$ be cyclic and separating unit vectors for
standard von-Neumann algebras $\clm_1$ and $\clm_2$ acting on Hilbert spaces $\clh_{\omega_1}$ and $\clh_{\omega_2}$ respectively. Let $\tau:\clm_1 \raro \clm_2$ be a normal unital completely positive map such that $\omega_2 \tau = \omega_1$. Then there exists a unique unital completely positive normal map 
$\tau':\clm_2' \raro \clm_1'$ ([section 8 in [OP] ) satisfying the duality relation 
\be 
\langle b\zeta_{\omega_2},\tau(a)\zeta_{\omega_2} \rangle =  \langle \tau'
(b)\zeta_{\omega_1},a\zeta_{\omega_1} \rangle 
\ee
for all $a \in \clm_1$ and $b \in \clm_2'$. For a proof, we refer 
to section 8 in the monograph [OP] or section 2 in [Mo1]. We set the dual 
unital completely positive map $\tilde{\tau}:\clm_2 \raro \clm_1$ defined by 
\be 
\tilde{\tau}(b) =  \clj_{\omega_1}\tau'(\clj_{\omega_2}b\clj_{\omega_2})\clj_{\omega_1}
\ee 
for all $b \in \clm_2$. In particular, we have $\omega_2 = \omega_1 \tilde{\tau}$. 

\vsp 
If $\cln=\clm_2$ is a von-Neumann sub-algebra of $\clm=\clm_1$ with a faithful normal state $\omega(a) = \langle \zeta_{\omega},a \zeta_{\omega} \rangle$ and 
$i_{\cln}:\cln \raro \clm$ is the inclusion map of $\cln$ into $\clm$. Then the dual of 
$i_{\cln}$ with respect to $\omega$, denoted by $\IE_{\omega}:\clm \raro \cln$ is a {\it norm one projection } i.e. 
\be 
\IE_{\omega}(abc)=a\IE_{\omega}(b)c
\ee 
for all $a,c \in \cln$ and $b \in \clm$ if and only if $\sigma_t^{\omega}(\cln)=\cln$. 
For a proof we refer to [Ta1] and for a local version of this theorem [AC].

\vsp 
We start with an elementary lemma. 

\begin{lem} 
Let $\clb$ be a unital $C^*$-algebra and $\tau:\clb \raro \clb$ be a unital completely positive map preserving a faithful state $\omega$. Then the following holds:

\NI (a) $\clb_{\tau}=\{x \in \clb: \tau(x)=x \}$ and $\clf_{\tau}= \{ x \in \clb: \tau(x^*x)=\tau(x)^*\tau(x),\;\tau(xx^*)=\tau(x)\tau(x)^* \}$
are $C^*$ sub-algebras of $\IM$ and $\clb_{\tau} \subseteq \clf_{\tau}$.

\NI (b) Let $\IE$ be norm-one projection on a $C^*$-sub-algebra $\clb_0 \subseteq \clb$ i.e. A unital completely positive map 
$\IE:\clb \raro \clb_0$ with range equal to $\clb_0$ satisfying bi-module property 
$$\tau(yxz)=y\tau(x)z$$ 
for all $y,z \in \clb_0$ and $x \in \clb$ and $\omega$ be a faithful invariant state for $\IE$. Then for some $x \in \clb$, $\IE(x^*)\IE(x)=\IE(x^*x)$ if and only if $\IE(x)=x$; 

\NI (c) $\clb_{\IE}=\clf_{\IE}$.   
\end{lem} 

\vsp 
\begin{proof} 
Though (a) is well known [OP], we include a proof in the following. The map $\tau$ being completely positive (2 positive is enough), we have Kadison inequality [Ka] $\tau(x^*)\tau(x) \le \tau(x^*x)$ for all $x \in \clb$ and equality holds 
for an element $x \in \clb$ if and only if $\tau(x^*)\tau(y)=\tau(x^*y)$ for all $y \in \clb$. This in particular shows that $\clf_{\tau}=\{x \in \clb:\tau(x^*x)=\tau(x^*)\tau(x),\;\tau(x)\tau(x^*)=\tau(xx^*) \}$ is $C^*$ -sub-algebra. We claim by the invariance property that $\omega \tau =\omega$ and faithfulness of $\omega$ and $\clb_{\tau} \subseteq \clf_{\tau}$. We choose $x \in \clb_{\tau}$ and by Kadison inequality we have 
$x^*x=\tau(x^*)\tau(x) \le \tau(x^*x)$ but $\omega(\tau(x^*x)-x^*x)=0$ by the invariance property and thus by faithfulness $x^*x=\tau(x^*)\tau(x)=\tau(x^*x)$. Since $x^* \in \clb_{\tau}$ once $x \in \clb_{\tau}$, we get $\tau(x)\tau(x^*)=\tau(xx^*)=xx^*$. This shows that $\clb_{\tau} \subseteq \clf_{\tau}$ and $\clb_{\tau}$ is an algebra by the polarization identity $4x^*y= \sum_{0 \le k \le 3}i^k(x+i^ky)^*(x+i^ky)$ for any two elements $x,y \in \clb$ as $\clb_{\tau}$ is a $*$-closed linear vector subspace of $\clb$. 

\vsp 
For (b) let $a \in \clb$ with $\IE(a^*a)=\IE(a^*)\IE(a)$. We have by the first part of the proof for (a), $\IE(a^*)\IE(b)=\IE(a^*b)$ for all $b \in \clb$ and thus
$$\omega(\IE(a^*)b))$$
$$=\omega(\IE(\IE(a^*)b))$$
(by the invariance property)
$$=\omega(\IE(a^*)\IE(b))$$
(by bi-module property)
$$=\omega(\IE(a^*b))$$
( by the first part of the argument used to prove (a) )
$$=\omega(a^*b)$$ 
for all $b \in \clb.$
So we have $\omega((\IE(a^*)-a^*)b)=0$ for all $b \in \clb$ and thus 
by faithful property of $\omega$ on $\clb$, we get $\IE(a^*)=a^*$ i.e. $\IE(a)=a$.

\vsp 
By (b), for $x \in \clb$ with $\IE(x^*)\IE(x)=\IE(x^*x)$ implies that $\IE(x)=x$ and thus 
$\cln_{\IE} \subseteq \cln_{\IE}$. Now by (a) we complete the proof for (c). 
\end{proof}

\section{ A characterization of local automorphism on $\IM=\otimes_{n \in \IZ}\!M^{(n)}_d(\IC)$ }

\vsp 
Let $\omega$ be a faithful state on a $C^*$ sub-algebra $\IM$ and $(\clh_{\omega},\pi_{\omega}, \zeta_{\omega})$ be the GNS representation of $(\IM,\omega)$ so that 
$\omega(x)=\langle \zeta_{\omega}, \pi_{\omega}(x) \zeta_{\omega} \rangle$ for all $x \in \IM$. Thus the unit vector $\zeta_{\omega} \in \clh_{\omega}$ is cyclic and separating for von-Neumann algebra $\pi_{\omega}(\cla)''$ that is acting on the Hilbert space $\clh_{\omega}$. Let $\Delta_{\omega}$ and $\clj_{\omega}$ be the modular and conjugate operators on $\clh_{\omega}$ respectively of $\zeta_{\omega}$ as described in (19). The modular automorphisms group $\sigma^{\omega}=(\sigma^{\omega}_t:t \in \IR)$ 
of $\omega$ is defined by 
$$\sigma_t^{\omega}(a)=\Delta^{it}_{\omega}a\Delta^{-it}$$
for all $a \in \pi_{\omega}(\IM)''$. 

\vsp 
We are interested now to deal with faithful states $\omega$ on $\IM=\otimes_{n \in \IZ} \!M_d^{(n)}(\IC)$. For a given faithful state $\omega$, $\cln_{\omega}=\pi_{\omega}(\IM_{\Lambda})''$ is a von-Neumann sub-algebra of $\clm_{\omega}=\pi_{\omega}(\IM)''$. 
In general, $(\sigma^{\omega}_t)$ may not keep $\cln_{\omega}$ invariant. However, for an infinite tensor product state $\omega_{\rho} = \otimes_{n \in \IZ} \rho^{(n)}$ with $\rho^{(n)}=\rho$ for all $n \in \IZ$, $\sigma_t^{\omega_{\rho}}$ preserves $\cln_{\omega}$. 

In such a case, $E_{\omega_{\rho},\Lambda}: \clm_{\omega_{\rho}} \raro \cln_{\omega_{\rho}}$ satisfy bi-module property (24) and furthermore, we have  
\be 
\IE_{\omega_{\rho},\Lambda}(\pi_{\omega_{\rho}}(x))=\pi_{\omega_{\rho}}(\IE_{\Lambda}(x))
\ee
for all $x \in \IM$, where $\IE_{\Lambda}$ is the norm one projection from $\IM$ onto $\ID_{\Lambda}$ with respect to the unique normalised trace $\omega_0$ of $\IM$. One can write explicitly 
$$\IE_{\Lambda} = \otimes_{n \in \IZ} \IE^{(n)}_{\Lambda},$$
where $\IE^{(n)}_{\Lambda}$ is the normalised trace $tr_0$ on $\IM^{(n)}_d(\IC)$ 
if $n \notin \Lambda$, otherwise $\IE^{(n)}_{\Lambda}$ is the identity map on $\!M^{(n)}_d(\IC)$ if $n \in \Lambda$.

\begin{lem} 
Let $\IM=\otimes_{k \in \IZ} \!M^{(k)}_d$ and $\IM_{\Lambda}$ be the local $C^*$ algebra associated with a subset $\Lambda$ of $\IZ$. Then there exists a norm projection $\IE_{\Lambda}:\IM \raro \IM_{\Lambda}$ 
preserving unique tracial state of $\IM$ and $\IE_{\Lambda}$ commutes with the group of automorphisms $\{\beta_g:g \in \otimes_{k \in \IZ} U_d(\IC) \}$, where $g=\otimes_n g_n$ with all $g_n=I_d$ except finitely many $n \in \IZ$. 
\end{lem}

\begin{proof} 
There exists a unique completely positive map $\IE_{\Lambda}:\IM \raro \IM_{\Lambda}$ satisfying 
\be 
tr(z\IE_{\Lambda}(x))=tr(zx)
\ee
for all $x \in \IM$ and $z \in \IM_{\Lambda}$. It follows trivially by a theorem of M. Takesaki [Ta1] since modular group being trivial preserves $\IM_{\Lambda}$. For an indirect proof, we can use duality argument used in [AC] to describe $\IE_{\Lambda}$ as the $KMS$-dual map of the inclusion map $i_{\Lambda}:z \raro z$ of $\IM_{\Lambda}$ in $\IM$. The modular group being trivial we get the simplified relation (26). That $\IE_{\Lambda}(z)=z$ for 
$z \in \IM_{\Lambda}$ is obvious by the faithful property of normalised trace. We may 
verify the bi-module property (24) directly: for all $y,z \in \IM_{\Lambda}$ and $x \in \IM$  
$$tr(z\IE_{\Lambda}(x)y)$$
$$=tr(yz\IE_{\Lambda}(x))$$
$$=tr(yzx)$$
$$=tr(zxy)$$
$$=tr(\IE_{\Lambda}(zx)y)$$
This shows that $\IE_{\Lambda}(zx)=z\IE_{\Lambda}(x)$ for all $z \in \IM_{\Lambda}$ and $x \in \IM$. By taking adjoint, we also get $\IE_{\Lambda}(xy)=\IE_{\Lambda}(x)y$ for all $y \in \IM_{\Lambda}$ and $x \in \IM$. Thus we arrive at the bi-module relation (24).  

\vsp 
Since $\beta_g(\IM_{\Lambda})=\IM_{\Lambda}$ for all $\Lambda$ for $g \in \otimes_{k \in \IZ} U_d(\IC)$, for all $z \in \IM_{\Lambda}$ and $x \in \IM$ we get 
$$tr(z \IE_{\Lambda} \beta_g(x))$$
$$=tr(z \beta_g(x))=tr(\beta_{g^{-1}}(z)x)$$
$$=tr(\beta_{g^{-1}}(z)\IE_{\Lambda}(x))$$
$$=tr(z \beta_g \IE_{\Lambda}(x))$$
So we get $\beta_g \IE_{\Lambda} = \IE_{\Lambda} \beta_g$ for all $g \in \otimes_{k \in \IZ}U_d(\IC)$. 
\end{proof} 

\vsp 
\begin{lem} 
Let $\beta$ be an $*$-automorphism on $\IM$ so that $\beta \theta=\theta \beta$. Then $\beta(\IM_R)=\IM_R$ if and only if $\beta=\beta_g$ for some $g \in SU_d(\IC)$.  
\end{lem}

\vsp 
\begin{proof} We claim that $\IM_R'=\{x \in \IM:xy=yx,\;y \in \IM_R\}$ is equal to $\IM_L$. For a proof we use conditional expectation $\IE_{\Lambda}$. Let $\Lambda_n$ be the subset $\{k:-n \le k \le n \}$ of $\IZ$ for each $n \ge 1$. We fix any $x \in \IM_R'$ and note that the element $\IE_{\Lambda_n}(x) \in \IM_{\Lambda_n}$ and 
$$y\IE_{\Lambda_n}(x)$$
$$=\IE_{\Lambda_n}(yx)$$
$$=\IE_{\Lambda_n}(xy)$$
$$=\IE_{\Lambda_n}(x)y$$
for all $y \in \IM_{R_n} \subset \IM_R$, where $R_n=\{k: 1 \le k \le n \}$.  
Thus $\IE_{\Lambda_n}(x) \in \IM_{L_n}=\IM_{\Lambda_n} \bigcap \IM_{R_n}'$, 
where $L_n=\{-n \le k \le 0 \}$. Since $||x-\IE_{\Lambda_n}(x)|| \raro 0$ as $n \raro \infty$, we conclude that $x \in \IM_{L}$.   

\vsp 
Since $\theta$ commutes with $\beta$, we get $\beta(\theta(\IM_R))=\theta(\beta(\IM_R))=
\theta(\IM_R)$. $\beta(\IM_L)=\beta(\IM_R')=\beta(\IM_R)'=\IM_R'=\IM_L$. This shows that
$\beta(\theta(\IM_R)\vee \IM_L)=\theta(\IM_R) \vee \IM_L$. Taking commutant we get
$\beta(M^{(1)}_d(\IC))=M^{(1)}_d(\IC)$. Thus $\beta=\beta_g$ for some $g$ on  $M_d^{(1)}(\IC)$.
Since $\beta$ commutes with $\theta$, we get $\beta=\beta_g$ on $\IM$.  
\end{proof}

\vsp 
\begin{lem}
Let $\omega$ and $\omega'$ be two translation invariant states of $\IM$. Two dynamical systems of endomorphisms $(\IM_R,\theta_R,\omega_R)$ and $(\IM_R,\theta_R,\omega'_R)$ are isomorphic if and only if
there exists a $g \in U_d$ such that $\omega'=\omega  \beta_g$.
\end{lem} 

\vsp
\begin{proof} 
Let $\alpha_R$ be a $C^*$-isomorphism on $\IM_R$ such that $\theta_R \alpha_R =\alpha_R \theta_R$ and $\omega'=\omega \alpha_R$. Commuting property clearly shows that $\alpha_R$ takes $C^*$-subalgebra $\theta(\IM_R)$ to itself. Thus we get an automorphism $\beta:\IM \raro \IM$, being the inductive limit of endomorphism on $C^*$-algebra $\IM_R \raro^{\theta_R} \IM_R$, such that
$\beta_{|}\IM_R=\alpha_R$ and $\omega'=\omega \beta$. Now by Lemma 3.2 we complete the 
proof.
\end{proof} 

\vsp 
\begin{cor} 
Mean entropy $s(\omega)$ of a translation invariant state $\omega$ is a dynamical invariant for the dynamics $(\IM_R,\theta_R,\omega_R)$. 
\end{cor}
\vsp 
\begin{proof} 
Since $s(\omega)=s(\omega \beta_g)$ for any $g \in SU_d(\IC)$, result follows by 
Lemma 3.3. 
\end{proof} 

\vsp 
\begin{lem}
Let $\beta:\IM \raro \IM$ be an automorphism such that $\omega'=\omega \beta$ on $\IM$ and $\beta  \theta =\theta  \beta$. If $\beta(\IM_R) \subseteq \theta^{k'}(\IM_R)$ for some $k' \in \mathbb{Z}$ then $\beta = \theta^k  \beta_g$ for some $g \in U_d(\IC)$ and 
$k' \le k$.
\end{lem} 

\vsp 
\begin{proof} 
We choose $k$ to be the largest possible value for $k'$ satisfying $\beta(\IM_R) \subseteq 
\theta^{k'}(\IM_R)$. We claim that $\beta(\IM_R)=\theta^k(\IM_R)$. Without loss of generality we assume that $k=0$, otherwise we reset $\beta$ to $\theta^{-k}\beta$ as new automorphism. Thus we have an automorphism $\beta$ on $\IM$ such that $\beta(\IM_R) \subseteq \IM_R$ and $\beta(\IM_R)$ is not a subset of $\theta(\IM_R)$. We need to show $\beta(\IM_R)=\IM_R$. For a $C^*$ sub-algebra $\cla$ of $\IM_R$, we write in the following $\cla'=\{x \in \IM_R: xy=yx,\;y \in \cla \}$ 
and $\cla_1 \vee \cla_2$ is used to denote $C^*$-algebra generated by $\cla_1$ and $\cla_2$. We consider the sub-algebra 
$$\IM^{(1)} = \beta(\IM_R) \bigcap \theta(\IM_R)' \subseteq M^{(1)}_d(\IC)$$ 
We claim that 
$$\beta(\IM_R) \bigcap \theta(\IM_R)' \neq \{zI^{(1)}_d:\;z \in \IC \}$$ 
Since otherwise we would have $\beta(\IM_R)' \vee \theta(\IM_R) = \IM_R$ and so 
$$(\beta(\IM_R)' \vee \theta(\IM_R)) \bigcap \theta(\IM_R)'=M_d^{(1)}(\IC)$$
However 
$$(\beta(\IM_R)'\vee \theta(\IM_R)) \bigcap \theta(\IM_R)' \supseteq \beta(\IM_R)' \bigcap \theta(\IM_R)'$$
$\beta(\IM_R) \vee \theta(\IM_R) \subseteq M_d^{(1)}(\IC)'= \theta(\IM_R)$ which says that $\beta(\IM_R) \subseteq \theta(\IM_R)$. This brings a contradiction and proves our claim.     

\vsp 
We also claim that $\IM^{(1)}$ is a factor. To that end we compute commutant of $\IM^{(1)}$ which is equal to $\beta(\IM'_R) \vee \theta(\IM_R)$ and any element in the centre of $\IM^{(1)}$ is also an element in $\beta(\IM'_R)$. Since $\IM_R$ is a factor, we conclude
that $\IM^{(1)}$ is a sub-factor of $M^{(1)}(\IC)$. 

\vsp 
Going along the same line, now we set sub-factor 
$\IM^{(n)}) = \beta(\IM_R) \bigcap \theta^n(\IM_R)' \subseteq M^{(1)}_d(\IC) \otimes M^{(2)}_d(\IC)..\otimes M^{(n)}_d(\IC)$. $M^{(n)}$ is a sub-factor of $\IM^{(n+1)}$ and we 
may write $\IM^{(n)}=\otimes_{1 \le k \le n} N^{(k)},$ 
where $N^{(k)}$ are sub-factors of $M^{(k)}(\IC)$ for each $k \ge 1$. 
We claim that $\vee _{n \ge 1} \IM^{(n)} = \beta(\IM_R)$. That $\vee_{n \ge 1} \IM^{(n)} \subseteq \beta(\IM_R)$ obvious. For the reserves inclusion, we fix any element $x \in \beta(\IM_R)$ and consider 
the sequence of elements $\IE_{\Lambda_n}(x):n \ge 1$, where $\Lambda_n= L \bigcup \{ 1,2,..,n \} \subset \IZ$. Since $\IE_{\Lambda}(\beta(\IM_R))= \beta(\IM_R) \bigcap \IM_{\Lambda}$ for any $\Lambda \subseteq \IZ$, we get 
$\IE_{\Lambda_n}(x) \in \beta(\IM_R) \bigcap \IM_{\Lambda_n} = \beta(\IM_R) \bigcap \theta^n(\IM_{\Lambda_R})'$. Since $||\IE_n(x)-x|| \raro 0$ as $n \raro \infty$, 
$x \in \vee_{n \ge 1} \IM^{(n)}$. Thus we deduce that  
$$\beta(\IM_R) = \otimes_{k \in R}N^{(k)}$$   
Furthermore, since $\theta(\beta(\IM_R))=\beta(\theta(\IM_R))$, we get 
$\theta(N^{(k)}) \subseteq N^{(k+1)}$. 

\vsp 
However the inductive limit $C^*$ algebra of $\beta(\IM_R) \raro^{\theta} \beta(\IM_R)$ is $\beta(\IM)$ equal to $\IM$. Thus we get $N^{(k)}=M^{(k)}(\IC)$ for all $k \ge 1$. This shows in particular $\beta(\IM_R)=\IM_R$. Thus result follows by Lemma 3.2. 
\end{proof}

\vsp 
\begin{thm}
Let $\beta$ be an $*$-automorphism on $\IM$ such that $\beta \theta=\theta \beta$. Then 
$\beta(\IM_{loc}) \subseteq \IM_{loc}$ if and only if $\beta= \theta^k \beta_g$ for some
$k \in \mathbb{Z}$ and $g \in U_d(\IC)$.  
\end{thm}

\vsp 
\begin{proof} 
If part is trivial. For `only if' part we fix an local automorphism $\beta$ on $\IM$.  Since it preserves local $*$-algebra, we have $\beta(\IM_{\{1\}}) \subseteq \IM_{[k,l]}$ for some $k \le l$ in $\mathbb{Z}$. Hence $\beta(\IM_R) \subseteq \IM_{[k,\infty)}$ by commuting property of $\theta$ with $\beta$. Now we use Lemma 3.5 to complete the proof. 
\end{proof}

\vsp 
\begin{thm} 
An automorphism $\beta:\clo_d \raro \clo_d$ commutes with canonical endomorphism $\lambda$ if and only if $\beta =\beta_g$ for some $g \in SU_d(\IC)$ on $\mbox{UHF}_d$. Furthermore, the following holds:

\NI (a) $\beta$ commutes also with $(\beta_z:z \in S^1)$ if and only if $\beta=\beta_g$ for some $g \in SU_d(\IC)$; 

\NI (b) Let $(\clo_d,\lambda,\psi)$ and $(\clo_d,\lambda,\psi')$ be two ergodic semi-group of endomorphisms so that their restrictions to $\mbox{UHF}_d$ be also ergodic. 
Then they are isomorphic if and only if there exists a $g \in U_d(\IC)$ such that 
$\psi'= \psi  \beta_g$.  
\end{thm}

\begin{proof} 
We claim that $\lambda(\clo_d)'=\{s_is_j^*:1 \le i,j \le d \}$. That $s_is_j^* \in \lambda(\clo_d)'$ is obvious by Cuntz relation. For 
converse we take an element $x \in \lambda(\clo_d)'$ i.e. $x\lambda(y)=\lambda(y)x$ for all $y \in \clo_d$. Then by Cuntz relation $s^*_ixs_jy=ys_i^*xs_j$ i.e. $s^*_ixs_j \in \clo'_d$. Since 
we have $\clo_d \bigcap \clo_d'=\IC$, we conclude that $s^*_ixs_j =\mu^i_jI$ for some $\mu^i_j \in \IC$. Thus $x = \sum \mu^i_js_is_j^*$. Similarly $\lambda^n(\clo_d)' = \mbox{span of} \{ s_Is_J^*:|I|=|J|=n \}$. So for an automorphism $\beta$ on $\clo_d$ which commutes with $\lambda$ we have $\beta(\lambda^n(\clo_d))=\lambda^n(\clo_d)$  and $\beta(\lambda^n(\clo_d)')=\lambda^n(\clo_d)'$ for all $n \ge 1$. Thus we have in 
particular $\beta(\mbox{UHF}_d)=\mbox{UHF}_d$ and $\beta=\beta_g$ for some $g \in SU_d(\IC)$ 
on $\mbox{UHF}_d$. So we have an automorphism $\beta'=\beta_{g^{-1}}  \beta$ on $\clo_d$ commuting with $\lambda$ and it is the identity map on its restriction to $\mbox{UHF}_d$. 

\vsp 
We claim that $\beta'=\beta_w$ for some $w \in S^1$ if $\beta$ and so $\beta'$ commutes with 
$\{\beta_z:z \in S^1)$. This follows by a simple application of Theorem 4.2 in [BE] since $\mbox{UHF}_d$ is the fixed point $C^*$-sub-algebra of $\clo_d$ of the compact abelian group action $(\beta_z:z \in S^1)$ [Cun]. This completes the proof of (a).      

\vsp 
For (b) let $\beta$ be an automorphism that commutes with $\lambda$ and $\psi=\psi' \beta$. Let $\omega$ and $\omega'$ be the restrictions of $\psi$ and $\psi'$ respectively to $\mbox{UHF}_d$. Then we have $\omega = \omega' \beta_{g'}$ on $\mbox{UHF}_d$ for some $g' \in SU_d(\IC)$ by first part of the statement of present corollary. Now we apply Proposition 7.6 in [BJKW] ( Proposition 2.6 in [Mo5] ) for existence of $w \in S^1$ such that $\psi'= \psi \beta_{g'} \beta_w$. We take $g= g' wI_d$ to complete the proof.        
\end{proof} 

\vsp 
We end this section with a proof of Proposition 1.1. 

\vsp 
\begin{proof} 
\vsp 
Let $\IM_0$ be a $C^*$-algebra of $\IM$ such that 
$\IM_0''=\IM_0$ and $\theta^{-1}(\IM_0) \subseteq \IM_0$. Then the mutually commuting family of $C^*$-algebras $\IM^{(n)} = \theta^n(\IM_0) \bigcap \theta^{n-1}(\IM_0)'$ are isomorphic copies of $\IM^0$ with $\IM^{(n)}= \theta^n(\IM^{(0)})$. 
We claim that 
$$\vee_{n \in \IZ} \IM^{(n)} = \IM$$ 
For a proof, let $x$ be an element in the commutant of $\vee_{n \in \IZ} \IM^{(n)}$. Let $\IE_0$ be the conditional
expectation onto $\IM_0$ with respect to the normalized trace on $\IM$. Then $\IE_0(x) \in \IM_0$ and for any $y \in \IM^{(0)}$ we have 
$$\IE_0(x)y$$
$$=\IE_0(xy)$$
$$=\IE_0(yx)$$
$$=y\IE_0(x)$$ 
So $\IE_0(x) \in \IM_0 \bigcap (\IM^{(0)})'$. 
We claim that $\IM_0 \bigcap (\IM^{(0)})' = \theta^{-1}(\IM_0)$. That $\theta^{-1}(\IM_0) \subseteq \IM_0 \bigcap (\IM^{(0)})' \subseteq \IM_0$ is obvious. But 
$$\IM_0 \bigcap (\IM^{(0)})'$$
$$= \IM_0 \bigcap \theta^{-1}(\IM_0)''$$
$$=\IM_0 \bigcap \theta^{-1}((\IM_0)'')$$
$$=\IM_0 \bigcap \theta^{-1}(\IM_0)$$
(since $\IM_0''=\IM_0$)
$$=\theta^{-1}(\IM_0)$$  
Thus $\IE_0(x) \in \theta^{-1}(\IM_0)$. 

\vsp 
Now we may repeat the argument with elements $y \in \IM^{(-1)}$ to show that $\IE_0(x) \in \theta^{-2}(\IM_0)$. Thus by mathematical induction we conclude that $\IE_0(x) \in \bigcap_{n \le 0} \theta^n(\IM_0)$. By our assumption (c), we get $\IE_0(x)$ is a scaler. 

\vsp 
Let $\IE_n$ be the conditional expectation from $\IM_0$ onto $\theta^n(\IM_0)$ with respect to the unique normalized trace. Then by the same argument use above, we get 
$\IE_n(x)$ is a scaler.  However, by (b), we have 
$$||\IE_n(x)-x|| \raro 0$$ 
as $n \raro \infty$. Thus we conclude that $x$ is a scaler. 
This shows $\vee_{n \in \IZ} \IM^{(n)} = \IM$ once relations (a)-(c) in the statement of Proposition 1.1 hold.

\vsp 
Since $\IM \bigcap \IM' = \IC$ and any element $x \in \IM^{(0)} \bigcap (\IM^{(0)})'$ commutes with all elements in $\IM^{(n)}$ for each $n \in \IZ$, we have $\IM^{(0)} \bigcap (\IM^{(0)})'= \IC$. We claim further that the linear map which extends the following map 
\be 
x=\Pi_{n \in \IZ} x_n \raro \tilde{x} =  \otimes_{n \in \IZ} \tilde{x}_n
\ee 
is a $C^*$ isomorphism between $\IM$ and $\otimes_{n \in \IZ } \tilde{\IM}^{(n)}$,  
where $x_n \in \IM^{(n)}$, taking values $I$ except for finitely many $n \in \IZ$ 
and $\tilde{\IM}^{(n)}$ are copies of $\IM^{(0)}$ with elements $\tilde{x}_n=\theta^{-n}(x_n) \in \IM^{(0)}$ for all $n \in \IZ$. The universal property of tensor products implies: as a vector space $\IM$ is isomorphic to $\otimes_{n \in \IZ} \tilde{\IM}^{(n)}$. That the map is also $C^*$ isomorphic, follows once we verify 
\be 
||x|| = ||\tilde{x}||
\ee 
for $x=\Pi_{n \in \IZ} x_n$. 
Note that $\IM$ being a nuclear $C^*$-algebra [ChE], there is a unique $C^*$ norm determined by its cross norm [Pau]. For commuting self-adjoint elements $(x_n)$, spectrum $\sigma(x)$ of $x=\Pi_{n \in \IZ}x_n$ is given by $\sigma(x)= \{ \Pi_{n \in \IZ} \lambda_n: \lambda_n \in \sigma(x_n) \}$, where $\sigma(x_n)$ is the spectrum of $x_n$. Thus 
$$||x||=\mbox{sup}_{\lambda_n \in \sigma(x_n)} \Pi_{n \in \IZ} |\lambda_n| $$
$$= \Pi_{n \in \IZ_n} \mbox{sup}_{\lambda_n \in \sigma(x_n) }|\lambda_n| $$
$$= \Pi_{n \in \IZ} ||x_n||$$

\vsp 
So for the equality of norms in (28), we can use Gelfand theorem on spectral radius for 
a self-adjoint element $x=\Pi_{ n \in \IZ}x_n $ and then use $C^*$ property of the norms to verify the equality (28) for all $x=\Pi_{n \in \IZ} x_n$. Thus the $C^*$ algebra $\IM$ is isomorphic to $\otimes_{n \in \IZ} \tilde{\IM}^{(n)}$, where we recall $\tilde{\IM}^{(n)}$ are copies of $\IM^{(0)}$.  

\vsp 
This shows that there exists an automorphism $\alpha:\IM \raro \IM$ that commutes with $\theta$ taking $\IM^{(n)}$ to $\tilde{\IM}^{(n)}$ for each $n \in \IZ$. In particular, $(\IM,\theta,tr_0)$ is isomorphic to $(\otimes_{n \in \IZ}\tilde{\IM}^{(n)},\theta,tr_0)$, where we used the same notations $tr_0$ for the unique normalized traces on $\IM$ and $\otimes_{n \in \IZ} \tilde{\IM}^{(n)}$ respectively. Since Connes-St\o rmer dynamical entropy is an invariant for $C^*$ dynamical system and is equal to mean entropy for the unique normalised trace, we get $\tilde{\IM}^{(n)} = \!M^{(n)}_{d}(\IC)$. This shows the isomorphism (20) is induced by an automorphism $\alpha:\IM \raro \IM$ such that
$\alpha(\IM_0)=\IM_L$, where $\alpha$ commutes with $\theta$ by our construction. 
\end{proof} 

\vsp 
\begin{rem} 
Proposition 1.1 says for a nuclear $C^*$ algebra $\IM$ with a normalised trace $\omega_0$, $C^*$-dynamical system $(\IM,\theta,\omega_0)$ satisfying (a)-(c) is isomorphic to a two sided translation on a quantum spin chain $C^*$-algebra. Furthermore, under the same hypothesis, a more general $C^*$-dynamical system $(\IM,\theta,\omega)$ is isomorphic to translation dynamics on a two sided infinite tensor 
product $\otimes_{n \in \IZ} \tilde{\IM}^{(n)}$, where $\tilde{\IM}^{(n)}$ are copies of a $C^*$ sub-algebra of $\IM$. 
\end{rem}

\section{Maximal abelian $C^*$ sub-algebras and automorphisms of $\IM$ }  

\vsp 
We recall briefly our notions used in the following text. Let $\Omega=\{1,2,..,d\}$ 
and $\Omega^{\IZ}=\times_{n \in \IZ} \Omega^{(n)}$, where $\Omega^{(n)}$ are copies of $\Omega$ and equip with product topology. Thus $C(\Omega^{\IZ})$ can be identified with the $C^*$-sub-algebra $\ID^e= \otimes_{n \in \IZ} D_e^{(n)}(\IC)$, where $D_e^{(n)}(\IC)=D_d(\IC)$ for all $n \in \IZ$ and $D_d(\IC)$ is the set of diagonal matrices with respect to an orthonormal basis $e=(e_i:1 \le i \le d)$ for $\IC^d$. In other words, 
we identify $d \times d$ diagonal matrices $D_d(\IC)$ with the algebra $C(\Omega)$ of complex valued continuous functions on $\Omega$. 

\vsp 
A commutative $*$-subalgebra $D$ of $\!M_d(\IC)$ is called maximal abelian if $D'=D$. In such a case $D=D_d(\IC)$ for some orthonormal basis $e$ for $\IC^d$. More generally, an abelian $C^*$ sub-algebra $D$ of a $C^*$ algebra $\clb$ is called maximal abelian if $D'=D$. As a first step, we investigate some maximal abelian $C^*$-subalgebras of $\IM=\otimes_{n \in \IZ}\!M^{(n)}_d(\IC)$ those are $\theta$ invariant i.e. $\theta(D)=D$. In the following we use norm one projections $\IE_{\Lambda}: \IM \raro \IM_{\Lambda}$ with respect to unique trace. We recall that $\IE_{\Lambda}$ commutes with $\{\beta_g:g \in \otimes U_d(\IC) \}$ by Lemma 2.2. 
 
\vsp 
\begin{lem} 
For a subset $\Lambda$ of $\IZ$ we have 

\NI (a) $\IM_{\Lambda}'=\IM_{\Lambda'}$;

\NI (b) $\IM_{\Lambda} \bigcap \ID^e = \ID^e_{\Lambda}$;

\NI (c) $\ID^e$ is a maximal abelian $*$-sub-algebra of $\IM$ and $\ID^e \equiv C(\Omega^{\IZ})$, the algebra of continuous functions on $\Omega^{\IZ}$;     

\NI (d) We set $C^*$-sub-algebra $\IM^e_{\Lambda} = \IM_{\Lambda} \vee \ID^e$ of $\IM$. Then the following hold:  

\NI (i) $(\IM^e_{\Lambda})'=\ID^e_{\Lambda'}$;

\NI (ii) $(\ID^e_{\Lambda'})' = \IM^e_{\Lambda}$. 

\end{lem} 

\vsp 
\begin{proof} 
For (a) we repeat the same idea used in the proof for $\IM_R'=\IM_L$ in Lemma 3.2. That 
$\IM_{\Lambda'} \subseteq \IM_{\Lambda}'$ is obvious. For the reverse inclusion, we fix any element $x \in \IM_{\Lambda}'$ and consider the elements $\IE_{\Lambda_1}(x)$ for 
$\Lambda \subseteq \Lambda_1$ so that $\Lambda_1 \bigcap \Lambda'$ is a finite subset of $\IZ$. 
For any unitary element $u \in \IM_{\Lambda}$ we have
$$u\IE_{\Lambda_1}(x)u^*$$
$$= \IE_{\Lambda_1}(uxu^*)$$
$$=\IE_{\Lambda_1}(x)$$
Thus $\IM_{\Lambda_1} \in \IM_{\Lambda_1} \bigcap \IM_{\Lambda}'=\IM_{\Lambda_1 \bigcap \Lambda'}$. Since $||x-\IE_{\Lambda_1}(x)|| \raro 0$ as $\Lambda_1 \uparrow \IZ$ in the sense of van-Hove [BR2], we conclude that $x \in \IM_{\Lambda'}$. 

\vsp 
For (b) we can repeat ideas of the proof of (a) with obvious modification. To prove the 
non trivial inclusion, we fix any element $x \in \IM_{\Lambda} \bigcap \ID^e$ and consider 
$\IE_{\Lambda_1}(x)$ for $\Lambda \subseteq \Lambda_1$ such that $\Lambda_1 \bigcap \Lambda'$ is a finite subset of $\IZ$. For any unitary element $u \in \IM_{\Lambda_1 \bigcap \Lambda'} \vee \ID^e_{\Lambda}$ 
we have 
$$u\IE_{\Lambda_1}(x)u$$
$$=\IE_{\Lambda_1}(uxu^*)$$
$$=\IE_{\Lambda_1}(x)$$
Thus $\IE_{\Lambda_1}(x) \in \IM_{\Lambda} \bigcap \ID^e_{\Lambda}=\ID^e_{\Lambda}$. 
Taking von-Hove limit $||x -\IE_{\Lambda_1}(x)|| \raro 0$ as $\Lambda_1 \uparrow \IZ$, we get $x \in \ID^e_{\Lambda}$.  This completes the proof of (b).

\vsp 
For non trivial inclusion $D_e' \subseteq D_e$, we fix an element $x$ in the commutant of $\ID^e$. Then we have 
$$u\IE_{\Lambda}(x)u^*$$
$$=\IE_{\Lambda}(uxu^*)$$
$$=\IE_{\Lambda}(x)$$ 
for all unitary element $u \in \ID^e_{\Lambda}= \otimes_{k \in \Lambda}D^{(k)}_e$ and thus $\IE_{\Lambda}(x) \in \ID^e_{\Lambda}$ since $\ID^e_{\Lambda}$ is a maximal abelian sub-algebra of $\IM_{\Lambda}$ for $|\Lambda| < \infty$. Since $||x-\IE_{\Lambda}(x)|| \raro 0$ as $\Lambda \uparrow \IZ$ as van Hove limit, we conclude that $x \in \ID^e$ as $\ID^e$ is the norm closure of $\{\ID^e_{\Lambda}: \Lambda \subset \IZ \}$.  

\vsp 
For (d), we have the following equalities: 
$$(\IM^e_{\Lambda})'=\IM_{\Lambda}' \bigcap (\ID^e)'$$
$$=\IM_{\Lambda'} \bigcap \ID^e $$
(by Lemma 4.1 (a) )
$$=\ID^e_{\Lambda'}$$ 
(by Lemma 4.1 (b) ), where $\Lambda'$ is the complementary set of $\Lambda$. 

\vsp 
Furthermore, we claim also that 
$$(\ID^e_{\Lambda'})'=\IM^e_{\Lambda}$$
That $\IM^e_{\Lambda} \subseteq (\ID^e_{\Lambda'})'$ is obvious. For the reverse inclusion, let $x \in (\ID^e_{\Lambda'})'$ and consider the sequence of elements $x_n=\IE_{\Lambda_n}(x)$, where $\Lambda_n$ is sequence of subsets containing $\Lambda$ and $\Lambda_n \uparrow \IZ$ as $n \raro \infty$. Then for all $y \in \ID^e_{\Lambda_n \bigcap \Lambda'}$, we have 
$$y\IE_{\Lambda_n}(x)$$
$$=\IE_{\Lambda_n}(yx)$$
$$=\IE_{\Lambda_n}(xy)$$
$$=\IE_{\Lambda_n}(x)y$$
and thus $\IE_{\Lambda_n} \in \ID^e_{\Lambda_n \bigcap \Lambda'} \bigcap \IM_{\Lambda_n} 
\subseteq \IM^e_{\Lambda}$. 
Taking limit $n \raro \infty$ in $||x-x_n|| \raro 0$, we get $x \in \IM^e_{\Lambda}$. 
\end{proof}

\vsp 
Let $\clp$ be the collection of orthogonal projections in $\IM$ i.e. $\clp = \{ p \in \IM: p^*=p,\;p^2=p \}$.  
So $\clp$ is a norm closed set and $\omega_{\rho}(e^I_I)={1 \over d^n}$ , where 
$e^I_I=..\otimes I_d \otimes |e_{i_1}\rangle\langle e_{i_1}| \otimes |e_{i_2}\rangle\langle e_{i_2}|..
\otimes |e_{i_n}\rangle\langle e_{i_n}| \otimes I_d ..$, where $e=(e_i:1 \le i \le d)$ is an orthonormal 
basis for $\IC^d$. This shows that $\clp$ has no non-zero minimal projection and range of 
the map $\omega_{\rho}:\clp \raro [0,1] $ at least has all $d-$adic numbers i.e. 
$\cli_d = \{{j \over d^n}: 0 \le j \le d^n,\;n \ge 1 \}$.       

\vsp 
An abelian sub-collection $\clp_0$ of projections is called {\it maximal } if there exists an abelian 
collection of projections $\clp_1$ containing $\clp_0$ then $\clp_1=\clp_0$. As a set $\clp_0$ is a closed subset of $\clp$. It is simple to verify that $\clp_0$ is maximal if and only if any projection commuting with elements in $\clp_0$ is an element itself in $\clp_0$. It is less obvious that the norm closure of the linear span of $\clp_0$ i.e. $\cls(\clp_0)$ is a maximal abelian $C^*$-subalgebra of $\IM$. In the following proposition we prove a simple result first as a preparation for a little more deeper result that follows next. 

\vsp 
\begin{lem} 
The set $\clp_e =\{p \in \ID^e: p^*=p,\;p^2=p \}$ is a maximal abelian set of projections in $\IM$. Furthermore 
$$\clp_e=\{ p \in \ID^e \bigcap \IM_{\Lambda}: p^*=p,\;p^2=p,\;\mbox{finite subset } \Lambda \subset \IZ \}$$    
\end{lem} 

\vsp 
\begin{proof} Let $q$ be a projection in $\IM$ that commutes with all the element of $\clp_e$. Then $q$ also 
commutes with $\ID^e$ and $\ID^e$ being maximal abelian by Lemma 4.1 we get $q \in \ID^e$. Thus $q \in \clp_e$
 
\vsp 
By Lemma 4.1 $\ID^e \equiv C(\Omega^{\IZ})$ and so a projection $p \in \ID^e$ can be identified with an indicator function of a close set. 
The product topology on $\Omega^{\IZ}$ being compact, any closed set is also compact. Thus 
any close set in $\Omega^{\IZ}$ is a finite union of cylinder sets. So 
$p \in \IM_{\Lambda}$ for a finite subset $\Lambda$ of $\IZ$ depending only on $p$.  
\end{proof}

\vsp 
We use the following notions $S^1=\{z \in \IC: |z|=1 \}$ and $(S^1)^d=S^1 \times S^1 ...\times S^1 (d \mbox{-fold})$ in the text. For an element $\ul{z} \in (S^1)^d$, we define automorphism $\beta_{\ul{z}}$ on $\!M_d(\IC)$ by 
\be 
\beta_{\ul{z}}(x)= D_{\ul{z}}x D_{\ul{z}}^*,
\ee
where $D_{\ul{z}}$ is the diagonal unitary matrix $((\delta^i_jz_i))$. We use notation 
$\beta^{(k)}_{\ul{z}}$ for automorphism acting trivially on $\IM$ except on $\!M_d^{(k)}(\IC)$ as $\beta_{\ul{z}}$. For $\ul{z}_k \in (S^1)^d,\;k \in \IZ$, 
$\beta=\otimes_{k \in \IZ}\beta^{(k)}_{\underline{z}_k}$ is an automorphism 
on $\IM$ and $\beta$ commutes with $\theta$ if $\ul{z}_k=\ul{z}$ for all $k 
\in \IZ$. We use notation $\beta_{\ul{z}}$ for $\otimes \beta^{(k)}_{\ul{z}}$. If
$\ul{z}=(z,z,..,z)$ then we simply use $\beta_z$ which is consistent with our notation 
automorphism $\beta_z$ on $\clo_d$. Note that $\beta_z=I$ on $\IM$ for any $z \in S^1$.  

\vsp 
\begin{lem} 
Let $\beta$ be an automorphism on $\IM$ such that $\beta=I$ on 
$\ID^e$. Then $\beta=\otimes_{k \in \IZ}\beta^{(k)}_{\underline{z}_k}$ 
on $\IM$ for some $\ul{z}_k \in (S^1)^d$. 
\end{lem} 

\vsp 
\begin{proof} We fix a finite subset $\Lambda$ of $\IZ$. For any $x \in \IM^e_{\Lambda}=\IM_{\Lambda} \vee \ID^e$, we have $xy=yx$ for all $y \in \ID^e_{\Lambda'}$. Thus 
$$\beta(x)y$$
$$=\beta(x)\beta(y)$$
$$=\beta(xy)$$
$$=\beta(yx)$$
$$=\beta(y)\beta(x)$$
$$=y\beta(x)$$
for all $y \in \ID^e_{\Lambda'}$. 
This shows that $\beta(\IM^e_{\Lambda}) \subseteq \IM^e_{\Lambda}$ since $\IM^e_{\Lambda} = (\ID^e_{\Lambda'})'$ by Lemma 4.1 (d). Furthermore, $\beta$ being an automorphism and 
$\beta^{-1}(x)=x$ for all $x \in \ID^e$, we also have $\beta^{-1}(\IM^e_{\Lambda}) \subseteq \IM^e_{\Lambda}$. Thus we have 
$$\beta(\IM^e_{\Lambda})=\IM^e_{\Lambda}$$ 
In particular, we have 
$$\beta(\IM_{\Lambda}) \vee \ID^e_{\Lambda'}= \IM^e_{\Lambda}$$ 
since $\beta$ fixes all elements in $\ID^e$. The last identity clearly shows that  $\beta(\IM_{\Lambda})=\IM_{\Lambda}$. Since $\beta$ fixes all elements in $\ID^e$ and so in particular all elements in $\ID^e_{\Lambda}$, we conclude that $$\beta=\beta^{(n)}_{\underline{z}_n}$$ 
for some $\underline{z}_n \in (S^1)^d$ on each $\!M^{(n)}_d(\IC),\;n \in \IZ$. Thus $\beta=\otimes_{n \in \IZ}\beta^{(n)}_{\underline{z}_n}$ on $\IM$ by multiplicative 
property of $\beta$.    
\end{proof} 

\vsp 
We arrive at the following hyper-rigidity property of maximal abelian 
$C^*$ sub-algebra $\ID^e$.  
\vsp 

\begin{cor}
If actions of two automorphisms $\alpha$ and $\beta$ are equal on $\ID^e$ then 
$$\beta  \alpha^{-1} = \otimes_{n \in \IZ} \beta^{(n)}_{\underline{z}_n}$$
 for some $\underline{z}_n \in (S^1)^d,\;n \in \IZ$.
\end{cor} 

\begin{proof} 
The automorphism $\beta^{-1} \alpha$ acts trivially on $\ID^e$ and thus
by Lemma 4.3 we get the required result.  
\end{proof} 

\vsp 
\begin{cor} 
Let $\beta$ be an automorphism on $\IM$ such that $\beta \theta = \theta \beta$. If $\beta(\ID^e)=\ID^e$ and $\beta \beta_g = \beta_g \beta$ on $\ID^e$ for all permutation $g$ of the basis $(e_i)$ then $\beta = \theta^k$ for some $k \in \IZ$. 
\end{cor} 

\vsp 
\begin{proof} 
Since $\beta$ is a continuous map on $\ID^e \equiv C(\Omega^{\IZ})$, we get a continuous map $\hat{\beta}:\Omega^{\IZ} \raro \Omega^{\IZ}$ so that $\beta(f(w))=f(\hat{\beta}(w))$ for all 
$w \in \Omega^{\IZ}$. The map $\hat{\beta}$ being one to and onto continuous, it takes open 
sets to open sets. In particular (finitely supported) cylinder sets to cylinder sets. Since $\beta$ is trace preserving, $\hat{\beta}$ is a Bernoulli measure preserving map. Thus $\hat{\beta}(E_0^{i})=E_{n(i)}^{\pi_0(i)}$, where $E_n^{i}$ is the cylinder set $\Omega \times ..\times \{ i \}^{\mbox{nth place}}  \times \Omega ...$ and $\pi_0(i) \in \{1,2,..,d\}$. Since $\beta$ commutes with permutation of basis $(e_i)$ on $\ID^e$ and $\theta$, we conclude that 
$n(i)=n$ and $\pi_0(i)=i$. This shows that $\beta=\theta^n$ on $\ID^e$ and hence 
$\beta  \theta^{-n} = \otimes_{n \in \IZ} \beta^{n}_{\ul{z}_n}$ by Corollary 4.4. 
Since both $\beta$ and $\theta_{-n}$ commutes with $\beta_g$ for any permutation 
matrices $g$ of the orthonormal basis $(e_i)$, we conclude that $\beta  \theta^{-n}=\beta_z$ for some $z \in S^1$. However $\beta_z$ on $\IM$ is $I$ for any $z \in S^1$
and thus conclude $\beta=\theta^n$ for some $n \in \IZ$.   
\end{proof} 

\vsp 
We are left to answer a crucial existence question on automorphisms, namely, given an auto-morphism $\beta_0:\ID^e \raro \ID^e$, is there an auto-morphism $\beta:\IM \raro \IM$ extending $\beta_0$? We begin with a lemma to that end.

\vsp 
\begin{lem} 
Let $\IP$ be a unital $C^*$ subalgebras of $\IM$ such that $\theta(\IP)=\IP$. Then 
$(\IP,\theta,\omega_0) \equiv (\otimes_{n \in \IZ} \!P^{(n)},\theta,\omega_0)$, where $\!P^{(n)}$ are isomorphic $C^*$-sub-algebra of $\!M_d^{(n)}(\IC)$ with $\theta(\!P^{(n)})=\!P^{(n+1)}$ and $\omega_0$ is the unique normalize trace on $\IM$.
\end{lem} 

\vsp 
\begin{proof} 
We set 
$$\IP_R = \IP \bigcap \IM_R = \IE_R(\IP)$$ 
and 
$$\IP_{R_n} = \IP \bigcap \IM_{R_n}=\IE_{R_n}(\IP),$$ 
where $R_n = \{ k \in \IZ: k > n \}$. We also set a sequence of increasing 
$C^*$ algebras defined by 
$$\IM_P^{(n)}=\IP_R \bigcap \IP'_{R_n}$$ 
for $n \ge 1$. We claim that 
$$\IP_R = \vee_{n \ge 1} \IM^{(n)}_P$$
We need to show the non trivial inclusion $\IP_R \subseteq \vee_{n \ge 1} \IM^{(n)}_P$. We fix any $x \in \IP_R$ and consider the sequence $x_n = \IE_{\Lambda_n}(x)$, where $\Lambda_n=\{ k \in \IZ: 1 \le k \le n \}$. Since $\IE_{\Lambda}(\IP_R)=\IP_R \bigcap \IM_{\Lambda}$, we get $x_n \in \IP_R \bigcap \IM_{\Lambda_n} \subseteq \IP_R \bigcap \IP'_{R_n}$, where we have used $\IP_{R_n} \subseteq \IM_{R_n}$ and so $\IM_{\Lambda_n} = \IM_{R_n}' \subseteq \IP'_{R_n}$. Since $||x-x_n|| \raro 0$ as $n \raro \infty$, we conclude that $x \in \vee_{n \ge 1} \IM^{(n)}_P$.  

\vsp 
Since $\theta(\IP)=\IP$, we get $\theta(\IP_R) = \IP \bigcap \theta(\IM_R) \subset \IP_R$. Thus the $C^*$-inductive limit of $\IP_R \raro ^{\theta_R} \IP_R$ is $\IP$ and  
$\IP = \vee_{n \in \IZ} \IP^{(n)}$
where $\IP^{(n)} = \theta^n(\IP_R) \bigcap \theta^{n+1}(\IP_R)'$ and $\IP^{(n)}$ is a commuting family of $C^*$ sub-algebra of $\IM$ with $\theta(\IP^{(n)})=\IP^{(n+1)}$. 

\vsp 
We consider the $C^*$ dynamical system $(\IP,\theta,\omega_0)$, where $\omega_0$ is the unique normalize trace on $\IM$. The Connes-St\o rmer dynamical entropy of $(\IP,\theta,\omega_0)$ is less then equal to the Connes-St\o rmer dynamical entropy of $(\IM,\theta,\omega_0)$ as $\IP \subseteq \IM$. 
The Connes-St\o rmer dynamical entropy being an invariant of $(\IP,\theta,\omega_0)$, is equal to the mean entropy i.e. 
$$S(\omega_0|\IP^{(0)}) \le S(\omega_0|\IM^{(0)})=ln(d),$$ 
where $S(\omega_0|\IP^{(0)})$ and $S(\omega_0|\IM^{(0)})$ are the total entropies of the states $\omega_0$ restricted to $\IP^{(0)}$ and $\IM^{(0)}$ respectively. Thus 
$\IP^{(0)}$ is isomorphic to a $C^*$-sub algebra $P^{(0)}$ of $\!M^{(0)}(\IC) \equiv \IM_d(\IC)$. 

\vsp 
We set $P^{(n)}=\theta^n(P^{(0)})$ and complete the proof that $\IP \equiv \otimes_{n \in \IZ}P^{(n)}$ by evoking the line of argument used in the proof of Proposition 1.1 to show that the map 
$$\pi_{n \in \IZ}x_n \raro \otimes_{n \in \IZ} \tilde{x}_n$$ 
is an isomorphism, where $i_0: x_0 \raro \tilde{x}_0$ is the isomorphism between 
$\IP^{(0)}$ and $P^{(0)}$ and $\tilde{x}_n= \theta^n(i_0(\theta^{-n}(x_n)))$ for all 
$x_n \in \IP^{(n)}$. 
\end{proof} 

\vsp 
\begin{thm}  
Let $\beta_0$ be an automorphism on $\ID^e$ commuting with $\theta$ then there exists an automorphism $\beta$ on $\IM$ that commutes with $\theta$. If $\alpha$ is an another automorphism extending $\beta_0$ then $\beta  \alpha^{-1}= \beta_{\underline{z}}$ for
some $\underline{z} \in (S^1)^d$.    
\end{thm}  

\vsp 
\begin{proof} 
For each $n \in \IZ$, we consider the orthogonal projections $f^k_k(n)=\beta_0(|e_k\rangle\langle e_k|^{(n)})):1 \le k \le d \}$ and the commutative $C^*$ algebra $\clc_0$ generated by elements $\{ \beta_0(|e_k\rangle\langle e_k|^{(n)}):1 \le k \le d,\;n \neq 0 \}$. We claim that $C^*$ algebra $\clc_0'$ is isomorphic to $\ID^e \vee \!M^{(0)}_d(\IC)$. 

\vsp 
By Lemma 4.2 the family of projections $(f^k_k(0):1 \le k \le d)$ are elements in $\IM_{\Lambda}$ for some finite subset $\Lambda$ of $\IZ$ and so $\clc'_0 \subseteq \IM_{\Lambda} \otimes \ID^e_{\Lambda'}$ 

\vsp 
Furthermore, we find some local automorphism $\alpha_{\Lambda}$ which acts trivially on $\IM_{\Lambda'}$ such that    
$$\alpha_{\Lambda}(\clc_0) = \theta^k(\ID^e_{\IZ_*}),\;\;\IZ_*= \{n \in \IZ, n \neq 0 \}$$ 
for some $k \in \IZ$. By taking the commutant of the equality, we get 
$$\theta_{-k}(\alpha_{\Lambda}(\clc_0')) = \ID^e \vee \!M_d^{(0)}(\IC),$$
where we have used Lemma 4.1 (d). So $\clc_0'$ is isomorphic to $\ID^e \vee \!M_d^{(0)}(\IC)$. 

\vsp 
Thus we can find elements $(f^k_j(0):1 \le k,j \le d)$ in $\clc_0'$ satisfying the matrix relations namely 
$$(f^i_j(0))^*=f^j_i(0): f^i_j(0)f^k_l(0)=\delta^k_jf^i_l(0)$$
extending the isomorphism $\beta_0:\ID^e \raro \ID^e$ to 
$$\beta_0:\ID^e \vee \!M^{(0)}_d(\IC) \raro \clc_0'$$ 
with 
$$\beta_0(e^i_j(0))=f^i_j(0),\;1 \le i,j \le d$$

\vsp 
We set $\beta$ on $\!M_d^{(n)}(\IC)$ by 
$$\beta(x)= \theta^n  \beta_0 \theta^{-n}(x)$$ 
We extend $\beta$ on $\IM_{loc}$ by linearity and $*$-multiplicative property. 
Let $\beta:\IM \raro \IM$ be the unique bounded extension of $\beta:\IM_{loc} \raro \IM$. Thus $\beta$ is a unital $*$-homomorphism.   

\vsp 
Since $\beta_0 \theta = \theta \beta_0$ on $\ID^e$, we verify for $x=\theta^n(x_0)$ with $x_0 \in \ID^e_{\{0\}}$ that 
$$\beta(x) = \theta^n \beta_0  \theta^{-n}  \theta^n(x_0)$$
$$=\theta^n \beta_0(x_0)$$
$$=\beta_0 \theta^n(x_0)$$
$$=\beta_0(x)$$     
Thus the map $\beta$, which is linear and $*$-multiplicative on $\IM$, in particular satisfies $\beta=\beta_0$ on $\ID^e$.  

\vsp 
We need to show $\beta$ is an automorphism on $\IM$. The multiplicative property of $\beta$ on $\IM$ says that $\cln=\{z :\beta(z)=0 \}$ is a two sided ideal of $\IM$. $\IM$ being a 
simple [Pow,ChE,BR1], $C^*$ algebra $\cln$ is either trivial null set or $\IM$. Since $\beta(I)=I$, we conclude that $\cln$ is the trivial null space i.e. $\beta$ is injective. In particular the map $x \raro ||\beta(x)||$ is a $C^*$-norm on $\IM$ where by $*$ homomorphism property of $\beta$ we verify that 
$$||\beta(x)^*\beta(x)||$$
$$=||\beta(x^*)\beta(x)||$$
$$=||\beta(x^*x)||$$  
for all $x \in \IM$.
However, $C^*$ norm being unique on a $*$-algebra if it exists, we get 
$||\beta(x)||=||x||$ for all $x \in \IM$. 

\vsp 
Furthermore, the map $\beta$ being norm preserving, $\beta(\IM)$ is norm closed and thus a $C^*$-sub algebra of $\IM$. Since $\beta \theta = \theta \beta$ on $\IM$, we also have $\theta(\beta(\IM))=\beta(\theta(\IM))=\beta(\IM)$. So the equality $\beta(\IM)=\IM$ 
follows by Lemma 4.6 once we show that $\beta(\IM)'=\IC$.  

\vsp 
Since $\beta(\ID^e)=\beta_0(\ID^e)=\ID^e$ and $\ID^e$ is maximal abelian, an element $x \in \beta(\IM)'$ is also in $\ID^e$ and thus $x=\beta_0(y)$ for some $y \in \ID_e$. Now we verify the following simple equalities for all $z \in \IM$: 
$$\beta(yz-zy)$$
$$=\beta(y)\beta(z)-\beta(z)\beta(y)$$
$$=\beta_0(y)\beta(z)-\beta(z)\beta_0(y)$$
$$=x\beta(z)-\beta(z)x$$
$$=0,$$
where $x \in \beta(\IM)'$. Thus we have $yz-zy=0$ for all $z \in \IM_{\Lambda}$ by the injective property of $\beta$. So we conclude that $y$ is a scaler multiple of identity. This shows $x=\beta_0(y)$ is also a scaler multiple of the identity element of $\IM$ i.e.  
$\beta(\IM)'=\IC$.
\end{proof}

\section{A complete isomorphism theorem for Bernoulli states of $\IM$ }

\vsp 
Let $\omega$ be a faithful state on the $C^*$ algebra $\IM=\otimes_{n \in \IZ} \!M_d^{(n)}(\IC)$ and $(\clh_{\omega},\pi_{\omega}, \zeta_{\omega})$ be the GNS representation of $(\IM,\omega)$ so that $\omega(x)=\langle \zeta_{\omega}, \pi_{\omega}(x) \zeta_{\omega} \rangle$ for all $x \in \IM$. Thus the unit vector $\zeta_{\omega} \in \clh_{\omega}$ is cyclic and separating for von-Neumann algebra $\pi_{\omega}(\cla)''$ that is acting on the Hilbert space $\clh_{\omega}$. Let $\Delta_{\omega}$ and $\clj_{\omega}$ be the modular and conjugate operators on $\clh_{\omega}$ respectively of $\zeta_{\omega}$ as described in section 2. We recall, the modular automorphisms group $\sigma^{\omega}=(\sigma^{\omega}_t:t \in \IR)$ 
of $\omega$ is defined by 
$$\sigma_t^{\omega}(a)=\Delta^{it}_{\omega}a\Delta^{-it}$$
for all $a \in \pi_{\omega}(\IM)''$. 

\vsp 
For a given faithful state $\omega$ of $\IM$, $(\sigma^{\omega}_t)$ may not keep $\cln_{\omega}$ invariant, where $\cln_{\omega}=\pi_{\omega}(\ID^e)''$ is a 
von-Neumann sub-algebra of $\clm_{\omega}=\pi_{\omega}(\IM)''$. However, for an 
infinite tensor product state $\omega_{\rho} = \otimes_{n \in \IZ} \rho^{(n)}$ with $\rho^{(n)}=\rho$ for all $n \in \IZ$, $\sigma_t^{\omega_{\rho}}= \otimes_{n \in \IZ}\sigma_t^{\rho_n}$ preserves $\cln_{\omega}$ if $\rho$ is a diagonal matrix in 
the orthonormal basis $e=(e_i)$ of $\IC^d$. In such a case, $E_{\omega_{\rho}}: \clm_{\omega_{\rho}} \raro \cln_{\omega_{\rho}}$ satisfy the bi-module property (24) 
and furthermore, we have  
\be 
\IE_{\omega_{\rho}}(\pi_{\omega_{\rho}}(x))=\pi_{\omega_{\rho}}(\IE_{\omega_0}(x))
\ee
for all $x \in \IM$, where $\IE_{\omega_0}$ is the norm one projection from $\IM$ onto $\ID^e$ with respect to the unique normalised trace $\omega_0$ of $\IM$. One can write explicitly 
$$\IE_{\omega_0} = \otimes_{n \in \IZ} \IE^{(n)}_0$$
where $\IE^{(n)}_0 = E_0$ for all $n \in \IZ$ and $E_0$ is the norm one projection 
from $\!M_d(\IC)$ onto the algebra of diagonal matrices $D_d(\IC)$ with respect to 
the basis $(e_i)$. The map $E_0:\!M_d(\IC) \raro \!D_d(\IC)$ is given by 
$$E_0(x) = \sum_{1 \le i \le d} \langle e_i, x e_i \rangle |e_i \rangle \langle e_i|$$
for all $x \in \!M_d(\IC)$. 

\vsp 
The map $\rho \raro \hat{\rho}$ is one to one and onto between the set of density matrices in $\!M_d(\IC)$ and the set of states on $\!M_d(\IC)$, where 
$$\hat{\rho}(x)=tr(x\rho)$$ 
for all $x \in \!M_n(\IC)$. We define entropy for a density matrix $\rho$ in $\!M_d(\IC)$ by 
$$S(\rho) = -tr(\rho ln \rho)$$
$$=-\sum_i \lambda_i ln(\lambda_i),$$ 
where $\rho= \sum_{1 \le i \le d} \lambda_i |e_i \rangle\langle e_i|$ 
for some $\lambda_i \in [0,1]$ with $\sum_i \lambda_i = 1$. One alternative 
description of $S(\rho)$ is given by a variational formula [OP]:   
$$S(\rho) = \mbox{inf}_{f=(f_i)} - \sum_i \hat{\rho}(|f_i\rangle\langle f_i|)ln (\hat{\rho}(|f_i\rangle\langle f_i|)),$$ 
where infimum is taken over all possible orthonormal basis. Thus the variational expression for $S(\rho)$ achieves its values for the basis $e=(e_i)$ that makes $\rho$ diagonal.   

\vsp 
\begin{thm} 
Mean entropy is a complete invariant of translation dynamics for the class of tensor product faithful states of $\IM$ i.e. Two infinite tensor product faithful states $\omega_{\rho}$ and $\omega_{\rho'}$ give isomorphic translation dynamics if and only if $S(\rho)=S(\rho')$. 
\end{thm}

\vsp 
\begin{proof} 
For the time being, we assume both the faithful states $\rho$ and $\rho'$ admit 
diagonal representations with respect to an orthonormal basis $e=(e_i)$ of $\IC^d$. 
So the restrictions of $\omega_{\rho}$ and $\omega_{\rho'}$ are Bernoulli states on $\ID^e$ with equal Kolmogorov-Sinai dynamical entropies. Thus by a theorem of D. Ornstein [Or1], we find an automorphism $\beta_0:\ID^e \raro \ID^e$ such that $\beta_0 \theta = \theta \beta_0$ on $\ID^e$ and $\omega_{\rho'}= \omega_{\rho} \beta_0$ on $\ID^e$. 

\vsp 
Let $\beta$ be an automorphism on $\IM$ that extends $\beta_0:\ID^e \raro \ID^e$ and commutes with $\theta:\IM \raro \IM$. Such an automorphism exists 
by Theorem 4.7.  

\vsp 
Now we consider the state $\omega= \omega_{\rho} \beta$ on $\IM$. The state $\omega$ is faithful and its modular automorphism group $\sigma_t^{\omega}:\pi_{\omega}(\IM)'' \raro \pi_{\omega}(\IM)''$ associated with the cyclic and separating vector $\zeta_{\omega}$ for $\pi_{\omega}(\IM)''$ satisfies the relation 
\be 
\sigma_t^{\omega} \hat{\beta} = \hat{\beta} \sigma_t^{\omega_{\rho}}
\ee 
on $\pi_{\omega_{\rho}}(\IM)''$, where $\hat{\beta}: \pi_{\omega_{\rho}}(\IM)'' \raro \pi_{\omega}(\IM)''$ is the automorphism defined by extending the map
\be 
\hat{\beta}: \pi_{\omega_{\rho}}(x) \raro \pi_{\omega}(\beta(x))
\ee   
for all $x \in \IM$. For a proof, we can use uniqueness of the automorphisms group of $\omega$ satisfying ${1 \over 2}$-KMS condition (21) [BR2].

\vsp 
Since automorphism $\beta$ is an extension of automorphism $\beta_0$ of $\ID^e$, the equation (31) says that $\sigma_t^{\omega}$ also preserves $\pi_{\omega}(\ID^e)''$. Thus once again by a theorem of M. Takesaki [Ta2], there exists a norm one projection 
$$\IE_{\omega}:\pi_{\omega}(\IM)'' \raro \pi_{\omega}(\IM)''$$ 
with range equal to $\pi_{\omega}(\ID^e)''$ satisfying 
$$\omega = \omega \IE_{\omega} $$ 
on $\pi_{\omega}(\IM)''$.

\vsp 
The rest of the proof are done in the follows elementary steps:

\vsp 
\NI (a) $\hat{\beta}^{-1} \IE_{\omega_{\rho}} \hat{\beta} = \IE_{\omega}$;

\vsp 
\begin{proof} 
It is obvious that both the maps $\IE_{\omega}$ and $\hat{\beta}^{-1} \IE_{\omega_{\rho}} \hat{\beta}$ are norm one projections satisfying bi-module property (24) from $\pi_{\omega}(\IM)''$ onto $\pi_{\omega}(\ID^e)''$. Furthermore, 
$$\omega \hat{\beta}^{-1} \IE_{\omega_{\rho}} \hat{\beta}$$ 
$$=\omega_{\rho} \IE_{\omega_{\rho}} \hat{\beta}$$
$$=\omega_{\rho} \beta$$
$$=\omega$$
Thus by uniqueness of the norm project with respect to $\omega$, we get the equality in (a). 
\end{proof} 

\vsp 
\NI (b) $\IE_{\omega} \pi_{\omega}(x) = \pi_{\omega}(\IE_{\omega_0}(x))$ for all $x \in \IM$. 

\begin{proof}
For any automorphism $\beta$, $\omega_0 \beta$ is also a normalised trace and thus by the uniqueness of normalised trace, we get $\omega_0 \beta = \omega_0$. Furthermore, $\beta^{-1} \IE_{\omega_0} \beta$ is also the norm one projection with respect to the unique 
normalised trace $\omega_0$ and thus we get $\beta^{-1} \IE_{\omega_0} \beta = \IE_{\omega_0}$ on $\IM$ by the uniqueness of the norm one projection with respect to $\omega_0$. We now complete the proof of (b) by the covariant property (a) and 
(30) valid for $\omega_{\rho}$.   
\end{proof} 

\vsp 
Let $\alpha_{p}$ be the automorphism on $\IM$ that extends linearly the map $x= \otimes_{n \in \IZ} x_n \raro \otimes_{n \in \IZ} x_{p(n)}$, where 
$p$ is a permutation on $\IZ$ that keeps all but finitely many points of $\IZ$ 
unchanged.

We also define automorphism $\hat{\alpha}_p:\pi_{\omega}(\IM)'' \raro \pi_{\omega}(\IM)''$ by extending the map 
$$\pi_{\omega}(x) \raro \pi_{\omega}(\alpha_p(x))$$
for all $x \in \IM$ to its weak$^*$-closer. 
In particular, we have 
$$\alpha_{p} \IE_{\omega_0} = \IE_{\omega_0} \alpha_{p}$$ on $\IM$ by (30). 

\vsp 
\NI (c) $\omega \alpha_{p} = \omega$ for all permutation $p$ on $\IZ$. 

\begin{proof} We verify the following equalities: 
$$\omega(x)$$
$$=\langle \zeta_{\omega},\pi_{\omega}(x) \zeta_{\omega}\rangle$$
$$=\langle \zeta_{\omega}, \IE_{\omega}(\pi_{\omega}(x)) \zeta_{\omega} \rangle$$
(since $\omega \IE_{\omega}=\omega$ )
$$=\langle \zeta_{\omega}, \hat{\alpha}_p ( \IE_{\omega}(\pi_{\omega}(x)) \zeta_{\omega} \rangle$$
(since $\omega \alpha_p = \omega$ on $\ID^e_Z$ as $\omega=\omega_{\rho'}$ on $\ID^e$ by 
our construction.) 
$$=\langle \zeta_{\omega}, \hat{\alpha}_p (\pi_{\omega}(\IE_{\omega_0}(x))) \zeta_{\omega} \rangle$$
( by (b) )
$$= \langle \zeta_{\omega}, \pi_{\omega}( \alpha_p(\IE_{\omega_0}(x))) \zeta_{\omega}$$
( by definition of $\hat{\alpha}_p$ )
$$= \langle \zeta_{\omega}, \pi_{\omega}( \IE_{\omega_0}(\alpha_p(x))) \zeta_{\omega} \rangle$$
( since $\IE_{\omega_0} \alpha_p = \alpha_p \IE_{\omega_0}$ )
$$= \langle \zeta_{\omega}, \IE_{\omega}(\pi_{\omega}(\alpha_p(x))) \zeta_{\omega} \rangle$$
( again by (b) )
$$= \langle \zeta_{\omega}, \pi_{\omega}(\alpha_p(x)) \zeta_{\omega} \rangle$$
( since $\omega = \omega \IE_{\omega}$ )
$$=\omega(\alpha_p(x))$$
\end{proof} 

\vsp 
\NI (d) $\omega$ is an infinite tensor product state;

\begin{proof} 
It follows from (c) once we use a quantum version of classical de Finetti theorem [St\o 2], which says any translation invariant factor state with permutation symmetry 
is an infinite product state.    
\end{proof} 

\vsp 
The state $\omega$ being an infinite product translation invariant state of $\IM$,   
we may write $\omega=\otimes_{n \in \IZ} \omega^{(n)}$, where $\omega^{(n+1)}=\omega^{(n)}$ for all $n \in \IZ$ and $\omega^{(0)}(x)=tr(x\rho_0)$ 
for some density matrix $\rho_0$ in $\!M_d(\IC)$. The mean entropy being equal to CS dynamical entropy for any infinite product state, we get $S(\rho_0)=S(\rho)$. 

\vsp 
Since $\omega = \omega_{\rho} \beta_0 = \omega_{\rho'}$ on $\ID^e$ by our construction, $\hat{\rho}_0(|e_k \rangle \langle e_k|)=\hat{\rho'}(|e_k \rangle \langle e_k|)$ for all $1 \le k \le d$. Thus 
$$S(\rho_0)$$
$$=S(\rho)$$
$$=S(\rho')$$
$$= - \sum_k \hat{\rho'}(|e_k \rangle \langle e_k|) ln \hat{\rho'}(|e_k \rangle \langle e_k|)$$
$$= - \sum_k \hat{\rho}_0(|e_k \rangle\langle e_k|)ln(\hat{\rho}_0(|e_k \rangle\langle e_k|))$$ 

\vsp 
By our starting assumption $\rho'$ is also diagonal with respect to the orthonormal basis $e=(e_i)$. Thus $\rho_0=\rho'$ once we show $\rho_0$ is also diagonal in the basis $e=(e_i)$.

\vsp 
\NI (e) For a density matrix $\rho_0$ in $\!M_d(\IC)$ we have 
\be 
S(\rho_0) = - \sum_i \hat{\rho}_0(|e_i \rangle \langle e_i|)ln(\hat{\rho}_0(|e_i \rangle \langle e_i|))
\ee 
if and only if $\rho_0 = \sum_i \lambda_i |e_i \rangle \langle e_i|$ for some $\lambda_i \in [0,1]$ i.e. $\rho_0$ admits a diagonal representation in the basis $e=(e_i)$. 

\vsp 
\begin{proof} 
Let $E_0$ be the norm one projection from $\!M_d(\IC)$ on $d$ dimensional diagonal matrices $D_d(\IC)$ with respect to the orthonormal basis $e=(e_i)$. Then equality in 
(33) says $S(\rho_0)=S(E_0(\rho_0))$ as 
$$\hat{\rho_0} \circ E_0(x)$$
$$=tr(\rho_0 E_0(x))$$
$$=tr(E_0(\rho_0)E_0(x))$$
$$=tr(E_0(\rho_0)x)$$
for all $x \in \!M_d(\IC)$. We compute now the relative entropy [OP] 
$$S(\hat{\rho}_0, \hat{\rho}_0 \circ E_0))$$
$$= - tr( \rho_0 ln (\rho_0 - E_0(\rho_0))$$
$$= S(\rho_0) + tr E_0(\rho_0 ln E_0(\rho_0))$$
$$(\mbox{ by the bi-module property of}\;E_0 )$$
$$= S(\rho_0) + tr E_0(\rho_0) ln (E_0(\rho_0))$$
$$= S(\rho_0) - S( E_0(\rho_0))$$
$$=0\;\;\mbox{by (33)}$$ 
Thus $\hat{\rho}_0 = \hat{\rho}_0 E_0$ i.e. $\rho_0=E_0(\rho_0)$. This completes the proof for `only if' part. If part is trivial.   
\end{proof} 

\vsp 
Thus we complete the proof by concluding $\hat{\rho}_0$ admits a diagonal representation in $e=(e_i)$ and so $\rho_0=\rho'$ since they equal values on $|e_i \rangle\langle e_i|$ for all $1 \le i \le d$ since $\omega = \omega_{\rho'}$ on $\ID^e$. Thus $\omega = \omega_{\rho_0}=\omega_{\rho'}$.

\vsp 
For more general faithful states $\rho$ and $\rho'$ of $\!M_d(\IC)$ satisfying $S(\rho)=S(\rho')$, we can find an element $g \in SU_d(\IC)$ such that $\rho$ and $\rho' \beta_g$ are simultaneously diagonalizable with respect to an orthonormal basis 
$(e_i)$. Since $S(\rho')=S(g\rho'g^*)$ and $\beta_g \theta = \theta \beta_g$ for 
any $g \in SU_d(\IC)$, we can complete the proof by the first part of the argument. 
This completes the proof.
\end{proof} 

\vsp 
Note that the faithful property of states of $\IM$ are also an invariant for translation dynamics and thus two infinite translation invariant product states with equal entropy need not give isomorphic dynamics. This observation is a contrast to classical situation where $\ID^e=C(\Omega^{\IZ})$ are isomorphic as a topological space for all values of $d \ge 2$, where $\Omega=\{1,2,..,d\}$. In contrast $C^*$ algebras $\otimes_{n \in \IZ} \!M^{(n)}_2(\IC)$ and $\otimes_{n \in \IZ}\!M^{(n)}_3(\IC)$ are not isomorphic [Gl].

\vsp 
\section{ A weak complete invariant for Kolmogorov states of $\IM=\otimes_{n \in \IZ}\!M^{(n)}_d(\IC)$ } 

\vsp   
\begin{pro} Let $\omega$ be a translation invariant pure state on $\IM=\otimes_{n \in \IZ} \!M^{(n)}_d(\IC)$ with Kolmogorov property. Then its Mackey index  of shift with respect to filtration generated by 
$E=[\pi_{\omega}(\IM_R)'\zeta_{\omega}]$  
is equal to $\aleph_0$. Same holds 
for filtration generated by the 
projection $F=[\pi_{\omega}(\IM_L)''\zeta_{\omega}]$ for a translation invariant state with Kolmogorov property. 
\end{pro} 

\vsp 
\begin{proof}  
Let $E$ be the support projection of $\omega$ in $\pi_{\omega}(\IM_R)''$ 
, where $(\clh_{\omega},\pi_{\omega},\zeta_{\omega})$ is the GNS space associated with the state $\omega$ on $\IM$, i.e. $E=[\pi_{\omega}(\IM_R)'\zeta_{\omega}]$. By Haag duality [Mo5] we also have $E=[\pi_{\omega}(\IM_R)'\zeta_{\omega}]=[\pi_{\omega}(\IM_L)''\zeta_{\omega}]$ and so $\theta(E)=[\pi_{\omega}(\theta(\IM_L))''\zeta_{\omega}]$. We claim that $\theta(E)-E$ is infinite dimensional. That $\theta(E)-E \neq 0$ follows as $\theta^n(E) \uparrow I$ 
as $n \uparrow \infty$ and $\theta^n(E) \downarrow |\zeta_{\omega}\rangle\langle\zeta_{\omega}|$ as $n \downarrow -\infty$. We choose a unit vector $g \in \theta(E)-E$ and thus in particular $g \perp \zeta_{\omega}$ since $\theta(E)\zeta_{\omega}=\zeta_{\omega}$. We set temporary notation $G=[\pi_{\omega}(\IM_R)'g]$ and note that $G \subseteq \theta(E)$ as $\IM_L \subseteq \theta(\IM_L)$ and $G \perp E$ as $g \perp \zeta_{\omega}$. Thus $G \le \theta(E)-E$. We claim that $G$ is an infinite dimensional subspace 
as otherwise we find a finite dimensional representation of $\IM_R$ on $G$ by restricting the representation $\pi_{\omega}$ of $\IM_L$ to $G$. That contradicts the simplicity of $\IM_L$ as a $C^*$-algebra [Mu].  

\vsp 
The proof uses Haag duality property but one may as well get a proof of Proposition 2.1 without using it as follows. In the following we use filtration generated by $F$ to give an alternative proof. 

\vsp 
Let $F=[\pi_{\omega}(\IM_L)''\zeta_{\omega}]$ be the projection in the GNS space 
$(\clh_{\omega},\pi_{\omega},\zeta_{\omega})$. Thus $\theta^{n+1}(F) \ge \theta^n(F)$ for all $n \in \IZ$ and $\theta^n(F) \uparrow I$ by cyclic property of the GNS representation. By Proposition 1.1, Kolmogorov property of $\omega$ is equivalent to 
$F_n=\theta^n(F) \downarrow |\zeta_{\omega}\rangle\langle\zeta_{\omega}|$ in strong operator topology as $n \downarrow -\infty$. Thus $F \ominus |\zeta_{\omega}\rangle\langle\zeta_{\omega}|$ rises to a system of imprimitivity with action of translation $\theta$ i.e. 
$$0 \le \theta^m(F_n-|\zeta_{\omega}\rangle\langle\zeta_{\omega}|) = F_{n+m}-|\zeta_{\omega}\rangle\langle\zeta_{\omega}|$$ and 
$F_n- |\zeta_{\omega}\rangle\langle\zeta_{\omega}| \downarrow 0$ 
as $n \downarrow -\infty$. By going along the same line of proof above, we choose a unite vector $g \in \theta(F)-F$ and use simplicity of the $C^*$-algebra $\IM_L$ to shows $G=[\pi_{\omega}(\IM_L)''g]$ is an infinite dimensional projection and thus the Mackey system generated by the projection $F-\zeta_{\omega}\rangle\langle\zeta_{\omega}|$ by shift is of index $\aleph_0.$ 
\end{proof}          

\vsp 
\begin{thm} 
Two translation invariant pure states on $\IM$ with Kolmogorov property gives unitarily equivalent shift on $\IM$.   
\end{thm} 

\begin{proof} 
Let $\omega,\omega'$ be two such states of $\IM$. By Mackey's theorem [Mac], as their Mackey's indices are equal to $\aleph_0$ by Proposition 6.1, we get 
a unitary operator $U:\clh_{\omega} \raro \clh_{\omega'}$ such that 
\be 
USU^*=S',\;\;UFU^*=F'
\ee 
In particular we get $U\Theta(X)U^*=\Theta(UXU^*)$ for all $X \in \pi_{\omega}(\IM)''$. 
\end{proof} 

\vsp 
\begin{thm} 
Let $\omega$ be a translation invariant pure state on $\IM$ with Kolmogorov property. Then $(\IM,\theta,\omega)$ is unitarily equivalent to a Bernoulli dynamics $(\IM,\theta,\omega_{\lambda})$, where $\lambda \in \IC^n$ is a unit vector. 
\end{thm} 

\vsp
\begin{proof} 
A pure Bernoulli state on $\IM$ admits Kolmogorov property and thus it is a simple case and proof follows by Theorem 6.2. 
\end{proof} 

\vsp 
It is already known [NeS] that Connes-St\o rmer dynamical entropy is zero for a pure translation invariant state. In general for a translation invariant state on $\IM$, $h_{CS}(\omega) \le s(\omega)$, where $s(\omega)$ is the mean entropy. One of the standing conjecture in the subject about equality. In particular it is not known yet whether $s(\omega)$ is zero for a translation invariant pure state. As a first step towards the conjecture it seems worth to investigate this question for such a state with Kolmogorov property. Note that mean entropy is yet to be interpreted as an invariant for the shift! However it is not hard to verify mean entropies of translation invariant states $\omega_1,\omega_2$ are equal if $(\IM,\theta,\omega_1)$ and $(\IM,\theta,\omega_2)$ are isomorphic provided the isomorphism $\beta$ is 
local i.e. preserves the algebra $\IM_{loc}=\bigcup_{\Lambda: |\Lambda| < \infty}\IM_{\Lambda}$. Thus in contrast to classical situation and Connes-St\o rmer's theory of Bernoulli shifts we have the following.  

\vsp 
\begin{thm} 
Let $(\IM,\theta,\omega)$ be as in Theorem 6.3. Then $(\IM,\theta,\omega)$ and $(\IM,\theta^k,\omega)$ are unitarily equivalent dynamics with Connes-St\o rmer dynamical entropy equal to zero but not isomorphic as $C^*$ dynamics for $k \neq 1$.     
\end{thm} 

\vsp 
\begin{proof} 
By Theorem 6.2 $(\IM,\theta,\omega)$ and $(\IM,\theta^k,\omega)$ are unitarily equivalent since each admits Kolmogorov's property with Mackey's index  equal 
to $\aleph_0$. For a translation invariant pure state, Connes-St\o rmer dynamical entropy is zero [NeS]. Since Kolmogorov states are also pure states [Mo2], their dynamical entropies are equal to zero.   
However Connes-St\o rmer dynamical entropy of $(\IM,\theta,\omega_0)$ is equal to mean entropy $s(\omega_0)$, where $\omega_0$ is the unique tracial state on $\IM$ 
and Conees-St\o rmer dynamical entropy of $(\IM,\theta^k,\omega_0)$ is equal to $ks(\omega_0)$. If these two dynamics were equivalent i.e. $\theta \beta = \beta \theta^k$ and $\omega \beta = \omega$ for some automorphism $\beta$
then we also would have $\omega_0 \beta = \omega_0$, $\omega_0$ being the unique trace of $\IM$. Since Connes-St\o rmer dynamical entropy is an invariant, as dynamics they are not equivalent since $s(\omega_0)=ln(d) \neq 0$.   
\end{proof}   

\vsp 
Any isomorphism on two compact Hausdorff Bernoulli spaces $\Omega^{\IZ}_d=\otimes_{n \in \IZ} \Omega^{(n)}_d$, $\Omega_d=\{1,2,..,d\}$ and $\Omega^{\IZ}_{d'}=\otimes_{n \in \IZ} \Omega^{(n)}_{d'}$, $\Omega_{d'}=\{1,2,..,d'\}$ equipped with the product topologies is local i.e. it takes any open set $\clo \in \Omega^{\IZ}_d$ with finite support to another open set with finite support in $\Omega^{\IZ}_{d'}$. An open set with finite support is a finite union of open cylinder sets. An open cylinder set is also compact in $\Omega^{\IZ}_d$ since $\Omega$ is a finite set and a continuous map takes a compact set to another compact set. However any subset of $\Omega^{\IZ}_{d'}$ can be covered by the collection of open cylinder sets and so in particular any compact subset of $\Omega^{\IZ}_{d'}$ can be covered by finitely many open cylinder sets of $\Omega^{\IZ}_{d'}$. This shows isomorphism is local. The isomorphism being local, we can use covariance relation with respect to translation action to conclude that isomorphism takes $\clf_{n]}$ to $\clf_{n+k]}$ for all $n \in \IZ$ for some fixed $k \in \IZ$. Similarly it takes backward filtration $\clf_{[n}$ to $\clf_{[n+k'}$ for all $n \in \IZ$ for some $k' \in \IZ$, where $\clf_{[n}$ is the minimal $\sigma-$field generated by the process $\{X_k: n \le k < \infty \}$. A family of classical Bernoulli states with entropy between $0$ and $ln(d)$ are embedded continuously into a Bernoulli state $\omega_{\lambda}$ on $\IM=\otimes_{n \in \IZ} \!M^{(n)}_d(\IC)$ and Kolmogorov-Sinai dynamical entropy is a complete invariant to determine a translation invariant maximal abelian von-Neumann sub-algebra upto isomorphism. On the other hand an isomorphism on $\IM= \otimes_{n \in \IZ} \!M^{(n)}_d(\IC)$ commuting with $\theta$ need not be local. We gave a precise answer to this question in Theorem 3.6.

\vsp   
In contrast for a pure Kolmogorov state $\omega$ the unitary operator $U:\clh_{\omega} \raro \clh_{\omega}$, intertwining weakly the two dynamics $(\IM,\theta,\omega)$ and $(\IM,\theta^2,\omega)$, does not preserve local algebra $\IM_{loc}$. Otherwise, $\IM$ being a simple $C^*$ algebra, we would have $U\pi_{\omega}(x)U^*=\pi_{\omega}(\beta(x))$ for some $C^*$-isomorphism $\beta:\IM \raro \IM$ by the faithful property of the representation $\pi_{\omega}$ and further inter-twinning relation $US^2=SU$ would have given $\theta \beta = \beta \theta^2$ on $\IM$. Since such a relation is impossible, we get a contradiction. A unitary conjugation that makes two Bernoulli dynamics $(\IM,\theta,\omega_{\lambda})$ and $(\IM,\theta^2,\omega_{\lambda})$ unitarily equivalent does not make them $C^*$-isomorphic. This feature is a deviation from an isomorphism map between two compact Hausdorff spaces, representing Bernoulli or Markov shifts [Or1]. Thus this feature is purely a quantum phenomenon and a strong deviation of our intuitions valid for classical spin chains. 

\vsp 
One natural question that arises here: when two translation dynamics $(\IM,\theta,\omega)$ and $(\IM,\theta,\omega')$ with Kolmogorov pure states $\omega,\omega'$ are isomorphic? It seems that this problem is far from being understood even in its primitive form. In the following we make a simple observation.  

\begin{pro}
Two pure Kolmogorov dynamical systems $(\IM,\theta,\omega)$ and $(\IM,\theta,\omega')$ are isomorphic if and only if there exists a unitary operator $U:\clh_{\omega} \raro \clh_{\omega}$ with $U\zeta_{\omega}=\zeta_{\omega}'$ such that the following holds:
 
\NI (a) $US_{\omega}=S_{\omega'}U;$

\NI (b) $U^*\pi_{\omega}(x)U= \pi_{\omega'}(\beta(x)),\; x \in \IM$ for
some automorphism $\beta$ on $\IM$. 

\end{pro}

\vsp 
Since $\IM$ is a UHF $C^*$-algebra and $\omega,\omega'$ are pure states of $\IM$, we can apply R.T. Power's theorem (Theorem 3.7 and Corollary 3.8 in [Pow]) 
to find an automorphism $\beta:\IM \raro \IM$ such that
$$\omega(x)=\omega'(\beta(x))$$ 
for all $x \in \IM$. The question of interest for $\omega \neq \omega'$ and $\beta$ is not equal to any order of translation i.e. $\beta \neq \theta^n$ for some $n \in \IZ$. There exists a unitary  operator $U:\clh_{\omega} \raro \clh_{\omega'}$ with $U\zeta_{\omega}=\zeta_{\omega}'$ so that 
$$U\pi_{\omega}(x)U^*=\pi_{\omega'}(\beta(x))$$

\vsp 
On the other hand by Theorem 6.2, there exists $U:\clh_{\omega} \raro \clh_{\omega'}$ such that $US_{\omega}=S_{\omega'}U$. Proposition 3.5 says that we need both the relations to hold simultaneously for some unitary $U$ for isomorphic dynamics.

\begin{proof} 
Assume (a) and (b). Then by (a) we get $U^*\pi_{\omega}(\theta(x))U=\pi_{\omega'}(\theta(\beta(x)))=\pi_{\omega'}(\beta(\theta(x))$. $\IM$ being a simple $C^*$ algebra, any non degenerate representation is faithful. Thus $\theta \beta = \beta \theta$ and $\omega' \beta =\omega$. The converse statement follows along the same line by tracing back the argument.   
\end{proof}

\section{Isomorphism theorem for pure Kolmogorov states}

\vsp 
\begin{pro} 
Let $\omega_c$ be an extremal element in the convex set of translation invariant states of
$\ID^e$. Then the non empty compact convex set 
$$\cls^{\theta}_{\omega_c}=\{ \omega:\IM \raro \IC,\;\mbox{a translation invariant state such that } \omega(x)=\omega_c(x) \forall x \in \ID^e \}$$ 
is a face in the convex set of translation invariant states of $\IM$ and extremal elements are ergodic states of $\IM$. 
\end{pro} 

\vsp 
\begin{proof} 
That $\cls^{\theta}_{\omega_c}$ is a non-empty compact set follows by Krein-Hahn-Banach theorem once clubbed with averaging method using $\theta$ and weak$^*$ compactness of the set of 
states of $\IM$. For the face property of $\cls^{\theta}_{\omega_c}$, let $\omega$ be an element in $\cls^{\theta}_{\omega_c}$. If $\omega= \lambda \omega_1 +(1-\lambda)\omega_0$ for some $\lambda \in (0,1)$ and translation invariant states $\omega_0,\omega_1$, then by extremal property of $\omega_1=\omega_0=\omega_c$ on $\ID^e$. Thus $S_{\omega_c}$ is a face and any extremal element in $S_{\omega_c}$ is also an extremal element in the convex set of translation invariant states of $\IM$. Thus $\omega$ is an ergodic state of $\IM$ by Theorem 4.3.17 in [BR1] since the dynamics $(\IM,\theta^n:n \in \IZ)$ is asymptotically abelian i.e. for all $x,y \in \IM$ we have $||x\theta^n(y)-\theta^n(y)x|| \raro 0$ as $|n| \raro \infty$. 
\end{proof} 

\vsp 
\begin{rem}
Any two extremal elements in $\cls^{\theta}_{\omega_c}$ are either same or orthogonal [ES]. More generally for two translation invariant classical ergodic states $\omega_c$ and $\omega'_c$, two extremal states $\omega \in \cls^{\theta}_{\omega_c}$ and $\omega' \in S_{\omega'_c}$ are either same or orthogonal. By taking restrictions of these two states $\omega$ and $\omega'$ to $\ID^e$, we arrive at Kakutani type of dichotomy theorem [Kak,ES] that says two extremal translation invariant states on $\ID^e$ is either same or orthogonal. However, Proposition 7.1 gives no information about existence of a translation invariant pure state in $\cls^{\theta}_{\omega_c}$. An ergodic translation invariant state $\omega$ need not be even a factor state in general. As an example we may take stationary Markov state $\omega_c$ on $\{1,2\}^{\IZ}$ with transition function that exchange position in each step it moves forward. By Power's criteria [Pow] we easily show non of the extremal elements in $\cls^{\theta}_{\omega_c}$ is a factor state. 

\vsp 
However, it is not clear, whether the cyclic property $[\pi_{\omega}(\ID^e)''\zeta_{\omega}]=[\pi_{\omega}(\IM)\zeta_{\omega}]$ is good enough for 
the factor and hence the purity property of an extremal element $\omega$ in $\cls^{\theta}_{\omega_c}$. In the next section, under some additional hypothesis 
$[\pi_{\omega}(\ID^e_{R})''\zeta_{\omega}]=[\pi_{\omega}(\IM_R)''\zeta_{\omega}]$, 
we will show that $\omega$ is a pure state.  
\end{rem} 

\vsp 
\begin{pro} 
Let $\omega$ and $\omega'$ be two translation invariant pure states of $\IM$
such that $\omega=\omega'$ on $\ID^e$. If $\zeta_{\omega}$ and $\zeta_{\omega'}$ are  cyclic vector for $\pi_{\omega}(\ID^e)''$ and $\pi_{\omega'}(\ID^e)''$ in $\clh_{\omega}$ and $\clh_{\omega'}$ respectively. Then $\omega= \omega'  \beta_{\underline{z}}$ for some $\underline{z} \in (S^1)^d$.  
\end{pro}  

\vsp 
\begin{proof} 
For two extremal elements $\omega$ and $\omega'$ for which $\zeta_{\omega}$ and $\zeta_{\omega'}$ are cyclic vectors for $\pi_{\omega}(\ID^e)''$ 
and $\pi_{\omega'}(\ID^e)''$ in $\clh_{\omega}$ and $\clh_{\omega'}$ respectively, 
we set unitary operator $V:\clh_{\omega} \raro \clh_{\omega'}$ defined by extending the 
linear map 
$$\pi_{\omega}(x)\zeta_{\omega} \raro \pi_{\omega'}(x)\zeta_{\omega'}$$ 
for $x \in \ID^e$. Thus we have 
\be 
S_{\omega'}V=VS_{\omega},
\ee
where $S_{\omega}$ and $S_{\omega'}$ are unitary operators on $\clh_{\omega}$ and $\clh_{\omega'}$ respectively associated with $\omega$ and $\omega'$ as 
defined in (3). Furthermore, we have 
$$V\pi_{\omega}(\IM)''V^*=\pi_{\omega'}(\IM)''$$  
and 
$$V\pi_{\omega}(x)V^*=\pi_{\omega'}(x),\;\; x \in \ID^e$$

\vsp 
In particular UHF$_d$ algebras $\cln_1=V\pi_{\omega}(\IM)V^*$ and $\cln_2=\pi_{\omega'}(\IM)$ are isomorphic to $\IM$ with their von-Neumann completions equal to  $\clb(\clh_{\omega'})$ and both $C^*$ algebras contain commutative $C^*$ algebra $\pi_{\omega'}(\ID^e)$. 

\vsp 
Since $\omega$ and $\omega'$ are pure, by Theorem 3.7 and Corollary 3.8 in [Pow], there exists a unitary operator $U$ such that 
\be 
U\cln_1U^*=\cln_2
\ee
and an automorphism $\beta$ on $\IM$ for which $\omega = \omega' \beta$ satisfying 
\be 
UV\pi_{\omega}(x)V^*U^*=\pi_{\omega'}(\beta(x)),\;\; x \in \IM
\ee 
However such a $U$ and so $\beta$ is not unique for two given pure states 
$\omega$ and $\omega'$. We include a proof for existence of $\beta$ satisfying (36) and (37) with $\beta(x)=x$ for all $x \in \ID^e$ in Appendix A. The proof uses the line of arguments used in the proof for Corollary 3.8 in [Pow] clubbing with the present special situation, where both $\cln_1$ and $\cln_2$ contains 
$C^*$ -sub-algebra $\pi_{\omega'}(\ID^e)$. 

\vsp 
Thus by Lemma 4.3, we have $\beta= \otimes_{k \in \IZ} \beta_{\underline{z_k}}$ on $\IM$ for some $\underline{z_k} \in (S^1)^d$. However $\omega = \omega \theta$ and $\omega' = \omega' \theta$. Thus $\omega = \omega' \beta_{\underline{z}}$ for some $\underline{z} \in 
(S^1)^d$ i.e. $\ul{z_k}=\ul{z} \in (S^1)^d$ for all $k \in \IZ$.     
\end{proof} 

\vsp 
\begin{cor} 
Let $\omega$ be a translation invariant pure state of $\IM$ and $\omega_c$ be its restriction to $\ID^e$ such that $\zeta_{\omega}$ is cyclic for $\pi_{\omega}(\ID^e)''$. Let $\beta_0$ be an auto-morphism on $\ID^e$ commuting with $\theta$ so that $\omega_c = \omega_c  \beta_0$ and $\beta$ be an automorphism on $\IM$ commuting with $\theta$ and extending $\beta_0$. Then we have 
$$\omega  \beta_{\ul{z}} = \omega   \beta$$ 
for an extension of $\beta_0$ to an automorphism $\beta$ on $\IM$. 
\end{cor} 

\vsp 
\begin{proof} 
We take $\omega'=\omega  \beta$ and find unitary operator $U:\clh_{\omega} \raro \clh_{\omega'}$ that maps $\pi_{\omega'}(x)\zeta_{\omega} \raro \pi_{\omega}(\beta(x))\zeta_{\omega}$. Since $\beta(\ID^e)=\ID^e$, we check that $\zeta_{\omega '}$ is also cyclic for $\pi_{\omega'}(\ID^e)$ in $\clh_{\omega'}$.  
Now we apply Proposition 7.3 to complete the proof.  
\end{proof} 

\vsp 
\begin{proof} (of Theorem 1.2) 
Let $\beta_0$ be an automorphism on $\ID^e$ so that $\beta_0 \theta = \theta \beta_0$
on $\ID^e$ and $\omega'_c= \omega_c \beta_0$. Let $\beta$ be an automorphism of $\IM$ extending $\beta_0$ of $\ID^e$ so that $\beta \theta = \theta \beta$. Such an automorphism is guaranteed by Theorem 4.7. Now we consider pure states $\omega'$ and $\omega \beta$ of $\IM$ which extends state $\omega'_0$ of $\ID^e$. Now by Proposition 7.3, we conclude that $\omega'= \omega \beta \beta_{\underline{z}}$ 
for some $\ul{z} \in (S^1)^d$. This completes the proof as $\beta  \beta_{\ul{z}}$ commutes with $\theta$.     
\end{proof}

\section{ Monic representation of Cuntz state and Kolmogorov state } 

\vsp 
A representation $\pi:\clo_d \raro \clb(\clh)$ is called {\it monic} if there is 
a cyclic vector $\zeta \in \clh$ for the maximal abelian $C^*$ sub-algebra $\ID^e$ i.e. the set  
$$\{ \pi(s_Is_I^*)\zeta: |I| < \infty \}$$ 
of vectors is total in $\clh$. A state $\psi$ of $\clo_d$ is called {\it monic } and its associated cyclic representation $(\clh_{\psi},\pi_{\psi},\zeta_{\psi})$ is called {\it monic representation } if $\{\pi_{\psi}(s_Is_I^*)\zeta_{\psi}: |I| < \infty \}$ is total in $\clh_{\psi}$. 

\vsp 
Given a monic state $\psi$ of $\clo_d$ we get a probability measure $\omega_{\mu}$ on 
$\Omega^R$ by defining its values on cylinder sets $E= \times_{1 \le k \le n} E_k$ by 
$$\omega_{\mu}(E_1 \times E_2 ...\times E_n \times \Omega \times \Omega ...)=\sum_{I=(i_1,...i_n): i_k \in E_k}\psi(S_IS_I^*)$$
for all subsets $E_k$ of $\Omega=\{1,2,..,d \}$. Cyclic property of $\zeta_{\psi}$ for 
the von-Neumann algebra $\{ \pi_{\psi}(s_Is_I^*):|I| < \infty \}''$ ensures that  
\be 
\omega_{\mu \circ \sigma_i} << \omega_{\mu},
\ee 
where
$$\sigma_i: (\Omega^{R},\clf_R,\omega_{\mu}) \raro (\Omega^{R},\clf_R,\omega_{\mu})$$
takes a word $(w_1w_2.....)$ in $\Omega^{R}$ to $(iw_1w_2....)$ in $\Omega^{R}$, $\clf_R$ 
is the Borel $\sigma-$field generated by the discrete topology of $\Omega^{R}$ and
$$
\frac{d\omega_{\mu  \circ \sigma_i}}{d\omega_{\mu}} = |f_i|^2
$$ 
for some functions $f_i \in L_2(\Omega^{R},\clf_R,\mu)$ with the property that
\be  
f_i(x) \neq 0 \;\; \mbox{for}\;\; \mu \;\; \mbox{-a.e}\; x \in \sigma_i(\Omega^{R})
\ee
In particular, the property (39) ensures that orthogonal isometries 
$$S^{\mu}_if(s)=f  \sigma_i(s)$$
satisfying Cuntz relations (15) and associates canonically an iterated function system [DJ] on the Borel measure space $(\Omega^{R},\clf_R,\omega_{\mu})$. Such an iterated function system $(\Omega^{R},\clf_R,\sigma_i,\mu)$ satisfying (38) and (39) is called [DJ] {\it monic system}. More generally, a monic representation $\pi$ of $\clo_d$ is unitary equivalent to the Cuntz representation given by {\it a monic system}  
$$\sigma_i: (\Omega^{R},\clf_R,\omega_{\mu}) \raro (\Omega^{R},\clf_R,\omega_{\mu})$$
$$S^{\mu}_if(s)=f \circ \sigma_i(s)$$
that associates canonically as an iterated function system [DJ] on the Borel measure space $(\Omega^{R},\clf_R,\sigma_i,\omega_{\mu})$ satisfying (38) and (39). 

\vsp 
Following [DJ], we say that a monic system is non-negative if $f_i \ge 0$ for all $i \in \Omega$. 

\vsp 
So far the state $\omega_{\mu}$ on $\ID^e_{R}$ need not be a right translation invariant. The following observation says a little more when $\omega_{\mu}$ is also right translation invariant.  

\vsp 
\begin{pro} Let $\omega$ be a translation invariant factor state of $\IM$ and $\psi$ be 
an extremal element in $K_{\omega}$. If $\psi$ is monic then the following statements are 
true:

\NI (a) $\zeta_{\omega}$ is a cyclic vector for $\pi_{\omega}(\ID^e)''$ in $\clh_{\omega}$;

\NI (b) $\omega$ is pure;

\NI (c) $\omega$ is Kolmogorov if and only if the restriction of $\omega$ to $\ID^e$ is Kolmogorov in the classical sense.    

\end{pro} 

\vsp 
\begin{proof} 
Since $\psi$ is monic, we have 
$[\pi_{\omega}(\IM_R)''\zeta_{\omega}]=[\pi_{\omega}(D^e_R)''\zeta_{\omega}]$ and $\omega$ being translation invariant we also have 
$[\pi_{\omega}(\theta^n(\IM_R))''\zeta_{\omega}]=[\pi_{\omega}\theta^n((\ID^e_{R}))''\zeta_{\omega}]$
Thus $[\pi_{\omega}(\ID^e)''\zeta_{\omega}]=\clh_{\omega}$ i.e (a) is true.  

\vsp 
The unit vector $\zeta_{\omega}$ being cyclic for $\pi_{\omega}(\ID^e)''$, $\pi_{\omega}(\ID^e)''$ is maximal abelian i.e. $\pi_{\omega}(\ID^e)''=\pi_{\omega}(\ID^e)'$ by Theorem 2.3.4 in [SS]. So for any element $X \in \pi_{\omega}(\IM)'$, we have in particular, $X \in \pi_{\omega}(\ID^e)'$ i.e. by maximal abelian property $X \in 
\pi_{\omega}(\ID^e)''$. Thus $X \in \pi_{\omega}(\IM)''$ and so $X$ is an element in the centre of $\pi_{\omega}(\IM)''$. By our hypothesis that $\omega$ is a factor state, we conclude $X$ is a scaler. This completes the proof for (b).  

\vsp 
For (c) we consider the dual state $\tilde{\psi}$ [Mo5] on $\tilde{\clo}_d$ defined by 
$$\tilde{\psi}(\tilde{s}_I\tilde{s}_J^*)=\psi(s_{\tilde{I}}s_{\tilde{J}}^*),$$ 
where $\tilde{I}=(i_n,i_{n_1},..i_1)$ if $I=(i_1,i_2,..,i_n)$. The map 
$$U: \pi_{\tilde{\psi}}(\tilde{s}_I\tilde{s}_j)\zeta_{\tilde{\psi}} \raro \pi_{\psi}(s_{\tilde{I}}s_{\tilde{J}})\zeta_{\psi}$$ 
extends to a unitary operator. It clearly shows that $\tilde{\psi}$ is monic if $\psi$ is so. Thus for a monic $\psi$, we have by (a) 
$$[\pi_{\omega}(\IM_L)''\zeta_{\omega}]=[\pi_{\omega}(\ID^e_L)''\zeta_{\omega}]$$

\vsp 
The state $\omega$ being pure by (b) and translation invariant, we have Haag duality property [Mo5] i.e. 
$$[\pi_{\omega}(\IM_R)'\zeta_{\omega}]= [\pi_{\omega}(\IM_L)''\zeta_{\omega}]$$

\vsp 
Thus 
$$E_n=[\pi_{\omega}(\theta^n(\IM_R))'\zeta_{\omega}]$$
$$=[\pi_{\omega}(\theta^n(\IM_L))''\zeta_{\omega}]$$
$$=[\pi_{\omega}(\theta^n(\ID^e_L)\zeta_{\omega}]$$ 
for each $n \in \IZ$. This shows $E_n \raro |\zeta_{\omega}\rangle\langle\zeta_{\omega}|$ as $n \raro -\infty$ if $\clf_n=[\pi_{\omega}(\theta^n(\ID^e_L)\zeta_{\omega}] \raro 
|\zeta_{\omega}\rangle\langle\zeta_{\omega}|$ as $n \raro -\infty$. 
\end{proof} 

\vsp 
For a given translation invariant Markov state $\omega_{\mu,p}$ on $C(\Omega^{\IZ})$ 
we consider the non-empty weak$^*$ compact convex set $\cls_{\omega_{\mu,p}}$ 
of translation invariant states of $\IM$ whose restriction on $C(\Omega^{\IZ})$ 
is $\omega_{\mu,p}$. While dealing with the isomorphism problem for states in $\cls_{\omega_{\mu,p}}$ we will restrict ourselves to $\mu,p$ that satisfies: 
\be 
\mu_i > 0,\;\;p^i_j > 0,\;\; \forall i,j \in \Omega
\ee
This additional condition on $\mu$ and $p$ 
ensures the following result: 

\vsp 
\begin{pro} 
Let $\omega_c=\omega_{\mu,p}$ with $\mu_i > 0$ and $p^i_j > 0$ for all $1 \le i,j \le d$ 
be a stationary Markov state on $C(\Omega^{\IZ})$ which is identified with maximal abelian $C^*$ sub-algebra $\ID^e$ of $\IM$. Then there exists a translation invariant pure state $\omega$ of $\IM$ extending $\omega_c$ such that any extremal element $\psi \in K_{\omega}$ is a monic state of $\clo_d$ and restriction of $\omega$ to $\IM_R$ 
is also pure. 
\end{pro} 

\vsp 
\begin{proof} 
We consider the monic representation $\pi:\clo_d \raro \clb(L^2(\Omega^{R},\clf_R, d\omega^R_c))$, which is define by 
$$\pi(s_j)f=f_j f \circ \sigma,$$
where 
$$f_j(x_1,x_2....)=\delta^j_{x_1}[{\mu_{x_2} \over \mu_jp^j_{x_2}}]^{1 \over 2}$$ 
associated with $\omega^R_c$, the restriction of $\omega_c$ to $C(\Omega^R)$,  
in Lemma 3.3 [DJ] and $\sigma$ is the left shift defined by $$\sigma((x_1,x_2,..)=(x_2,x_3,..)$$ 
We set $\lambda$-invariant state of $\clo_d$ by 
$$\psi(s_Is_J^*)=\langle \zeta_{\psi},\pi(s_Is_J^*)\zeta_{\psi}\rangle$$ 
where $\zeta_{\psi}$ is the constant function identically equal to $1$. Theorem 3.6 in [DJ] says that $\psi$ is a pure state of $\clo_d$ and Corollary 3.3 in [DJ] says that $\psi$ is also a monic state.  

\vsp 
Let $\omega$ be the unique translation invariant state $\omega$ of $\IM$ for which $\omega_R$ is equal to the restriction of $\psi$ to $\mbox{UHF}_d$ algebra 
$\{s_Is_J^*:|I|,|J| < \infty \}$. That $\omega$ is pure follows 
by Theorem 2.6 in [Mo4] since $\omega_R$ is pure ( in particular type-I factor state ).  

\vsp 
Since any other extremal element $\psi'$ in $K_{\omega}$ is given by $\psi'=\psi \circ \beta_z$ for some $z \in S^1$ [BJKW, Proposition 7.6], we conclude $\psi'$ is also a monic state of $\clo_d$. 
\end{proof}

\begin{cor} 
Let $\omega_{\mu,p}$ and $\omega$ be as in Proposition 8.2. Then the following statement hold:

\NI (a) $\omega_{\mu,p}$ is Kolmogorov; 

\NI (b) $\zeta_{\omega}$ is cyclic for $\pi_{\omega}(\ID^e)''$ in $\clh_{\omega}$; 
 
\NI (c) $\omega$ is pure and Kolmogorov. 
  
\end{cor} 

\vsp 
\begin{proof} 
A classical state is Kolmogorov if and only if dynamical entropy of $(\ID^e,\theta,\omega_{\mu,p})$ is strictly positive by Rokhline-Sinai Theorem [Pa]. We may recall dynamical entropy [Pa] of a stationary Markov state $\omega_{\mu,p}$ is 
$h_{\omega_{\mu,p}}(\theta)= -\sum_{i,j} \mu_i p^i_j ln(p^i_j)$ which is strictly positive to complete proof for (a). 

\vsp 
Now (b) and (c) follows by Proposition 8.1 since any extremal element $\psi$ in $K_{\omega}$ is monic by Proposition 8.2.   
\end{proof} 

\vsp 
We have the following result as an application of Theorem 1.2 and Proposition 8.2

\vsp 
\begin{thm} 
Let $\omega_{\mu,p}$ and $\omega_{\mu',p'}$ be two isomorphic 
stationary Kolmogorov states on $C(\cld_d^{\IZ})$ satisfying condition (40). Then there exists two Kolmogorov states $\omega \in \cls_{\omega_{\mu,p}}$ and $\omega' \in \cls_{\omega_{\mu',p'}}$ respectively 
giving isomorphic translation dynamics on $\IM$.
\end{thm}

\vsp 
Theorem 1.3 and Theorem 8.4 puts some more relevant future questions on general mathematical structure of a Kolmogorov state of $\IM$. Since mean entropy of 
a pure product state $\omega$ of $\IM$ is zero and for a given 
$\omega_{\mu,p}$ of $\ID^e$ we may fix a pure product state $\omega$ of $\IM$ such that the Kolmogorov-Sinai entropies of $\omega_c$ and $\omega_{\mu,p}$ equal, 
we could prove $s(\omega')$ is zero for any Kolmogorov state $\omega'$ of $\IM$ 
that extends $\omega_{\mu,p}$ provided we had known that the mean entropy is an 
invariant for translation dynamics i.e. does not change under isomorphism. On the contrary $s(\omega') \neq 0$ would have proved that the mean entropy is not an invariant for translation dynamics. 

\vsp 
In complete mathematical generality for two Kolmogorov states $\omega$ and $\omega'$ of $\IM$ The following question remains to be answered.  

\vsp  
\NI Q. Given a Kolmogorov state $\omega$ of $\IM$, does it admit a translation invariant maximal abelian $C^*$-sub algebra $\ID$ of $\IM$ and an automorphism $\alpha$ on $\IM$ satisfying the following statements: 

\vsp 
\NI (a) $\ID=\alpha(\ID^e)$ and $\alpha \theta = \theta \alpha$;

\vsp 
\NI (b) $\omega \alpha = \omega_{\mu,p}$ on $\ID^e$ for some stationary Markov state $\omega_{\mu,p}$.  

\vsp 
The main interesting point here, a translation invariant state $\omega_c$ of $C(\Omega^{\IZ})$ need not be a classical Markov state $\omega_{\mu,p}$ [Or2] as the dynamics $(C(\Omega^{\IZ}),\theta,\omega_c)$ may not admit a square root  i.e. another automorphism $\beta$ on $C(\Omega^{\IZ})$ such that $\beta^2=\theta$ and $\omega_c \beta =\omega_c$. Nevertheless a Kolmogorov state $\omega$ of $\IM$ extending $\omega_c$ may admit a square root i.e. an automorphism $\beta:\IM \raro \IM$ with $\theta=\beta^2$, need not keep the same maximal abelian $C^*$-subalgebra $\ID^e$ invariant.

\section{ Appendix A } 

\vsp 
\begin{thm} 
Let $\pi$ and $\pi_0$ be two unital representations of $\IM$ in $\clb(\clh)$ with $\clh$ separable satisfying the following conditions:   

\NI (a) $\pi(\IM)''= \pi_0(\IM)''=\clb(\clh)$;  

\NI (b) $\pi(x)=\pi_0(x)$ for all $x \in \ID^e$. 

\vsp 
Then there exists an automorphism $\beta$ on $\IM$ such that $\pi_0 = \pi \beta$ for which $\beta(x)=x$ for all $x \in \ID^e$.  
\end{thm} 

\vsp 
\begin{proof} 
We fix any finite subset $\Lambda$ of $\IZ$ and consider $C^*$ sub-algebra $\IM^e_{\Lambda'} = \IM_{\Lambda'} \vee \ID^e=\IM_{\Lambda'} \vee \ID^e_{\Lambda}$, where $\Lambda'$ is the complementary set of $\Lambda$ in $\IZ$. 

\vsp 
We claim that 
\be 
\pi(\IM^e_{\Lambda'})''=\pi_0(\IM^e_{\Lambda'})''
\ee

\vsp 
We have shown in Lemma 4.1 (d) that $\IM^e_{\Lambda'}=(\ID^e_{\Lambda})'$. So $$\pi(\IM^e_{\Lambda'})''= \pi(\IM_{\Lambda'})'' \vee \pi(\ID^e_{\Lambda})''$$ 
Since $\pi(\IM_{\Lambda})''$ is a finite type-I factor of $\clb(\clh)$, we have $$\pi(\IM_{\Lambda})''=\pi(\IM_{\Lambda'})'$$ 
So 
$$\pi(\IM^e_{\Lambda'})'$$
$$=\pi(\IM_{\Lambda'})'\bigcap \pi(\ID^e_{\Lambda})'$$
$$=\pi(\IM_{\Lambda})'' \bigcap \pi(\ID^e_{\Lambda})'$$
$$= \pi(\ID^e_{\Lambda})''$$ 
Similarly 
$$\pi_0(\IM^e_{\Lambda'})'=\pi_0(\ID^e_{\Lambda})''$$ 

\vsp 
Since $\pi$ and $\pi_0$ agrees on $\ID^e$ and so in particular on $\ID^e_{\Lambda}$, we get the required equality in (41). Since $\IM^e_{\Lambda'}$ is a hyper-finite $C^*$ algebra 
equality in (41) holds, Theorem 3.7 in [Pow] says that there exists an automorphism $\beta_{\Lambda}$ on $\IM^e_{\Lambda'}$ such that 
$$\pi_0 = \pi  \beta_{\Lambda}$$ 
Since the automorphism $\beta_{\Lambda}$ preserves the centre of $\IM^e_{\Lambda'}$, which is equal to $\ID^e_{\Lambda}$, we get $\beta_{\Lambda}(\ID^e_{\Lambda})=\ID^e_{\Lambda}$. 

\vsp 
Furthermore, $\IM$ being a simple $C^*$-algebra, the representations $\pi$ and $\pi_0$ of $\IM$ is faithful. In particular, since $\pi$ and $\pi_0$ agrees on $\ID^e$, we get $$\pi(\beta_{\Lambda}(x))$$
$$=\pi_0(x)$$
$$=\pi(x)$$
for all $x \in \ID^e_{\Lambda}$. Thus by the faithful property of $\pi$, we get 
$$\beta_{\Lambda}(x)=x$$ 
for all $x \in \ID^e_{\Lambda}$. Furthermore, $\pi(\IM_{\Lambda})$ and $\pi_0(\IM_{\Lambda})$ are isomorphic finite $C^*$ algebra and thus we find an extension of $\beta_{\Lambda}:\IM^e_{\Lambda'} \raro \IM^e_{\Lambda'}$ to an automorphism, still denoted by the same symbol $\beta_{\Lambda}:\IM \raro \IM$ for which we have the following properties: 

\NI (a) $\beta_{\Lambda}(x)=x$ for all $x \in \ID^e_{\Lambda}$;

\NI (b) $\beta_{\Lambda}(\IM_{\Lambda})=\IM_{\Lambda}$ and $\beta_{\Lambda}=\beta_{\ul{z}_i}$ on $\IM_{\Lambda}$ for some $z_i \in S^1,\;i \in \Lambda$; 

\NI (c) $\pi_0 = \pi  \beta_{\Lambda}$ on $\IM$.   

\vsp 
We choose a sequence of finite subset $\Lambda_n \uparrow \IZ$ as $n \uparrow \infty$ in van Hove sense ( we may take the sequence to be $\Lambda_n = \{ k \in \IZ: -n \le k \le n \}$ ) and for a finite subset $\Lambda$ of $\IZ$, we consider the restrictions of 
$\beta_{\Lambda_n}$ to $\IM_{\Lambda_n} \raro \IM_{\Lambda_n}$ that satisfies (a) to (c). 

\vsp 
By (c), we get $\pi_0(x) = \pi \beta_{\Lambda_n}(x) = \pi \beta_{\Lambda_{n+1}}(x)$ for all 
$x \in \Lambda_n$. By faithful property of $\pi$, we get 
\be 
\beta_{\Lambda_n}(x)=\beta_{\Lambda_{n+1}}(x)
\ee 
for all $x \in \IM_{\Lambda_n}$. Thus we may set a map $\beta:\IM_{loc} \raro \IM_{loc}$ defined by 
$$\beta(x)=\beta_{\Lambda_n}(x)$$ 
for $x \in \IM_{\Lambda_n}$. The map $\beta$ is well defined and $*$-homomorphism from 
$\IM_{loc}$ onto $\IM_{loc}$ by the consistency relation (42). Thus it has a unique bounded extension still denoted by $\beta:\IM \raro \IM$ such that $\beta(\IM)=\IM$. That $\beta$ is injective follows by simple property of $\IM$ as the null space $\cln= \{x \in \IM: \beta(x)=0$ is a two sided ideal by $*$-homomorphism property of $\beta$. Thus $\beta$ is indeed an automorphism on $\IM$ satisfying $\pi=\pi_0 \beta$ on $\IM$ with $\beta(x)=x$ for all $x \in \ID^e$. This completes the proof. 

\end{proof}   

\section{ Appendix B} 

\vsp 
Here we give a sketch leaving details as it requires quite a different involved framework to prove 
our claim that the unique ground state $\omega_{XY}$ of $H_{XY}$ [AMa] model when restricted 
to one side of the chain $\IM_R$ ( or $\IM_L$ ) is faithful. We refer to Chapter 6 of the 
monograph [EvKa] for details on the mathematical set up for the Fermion algebra and what follows now
is based on [Ma]. Let $\clh$ 
be a complex separable infinite dimensional Hilbert space $(\clh=L^2(\!R))$ and consider the 
universal simple unital $C^*$-algebra $\cla_F$ generated by $\{a(f): f \in \clh \} $ over $\clh$, where 
$a:\clh \raro \cla_F$ is a conjugate linear map
satisfying 
$$a(f)a(g)+a(g)a(f)=0$$
$$a(f)a(g)^*+a(g)^*a(f)=\langle g,f \rangle 1$$ 
For an orthonormal basis $\{e_j:j \in \IZ \}$ for $\clh$ we fix a unitary  operator 
$U:e_j \raro e_j$ if $j \ge 1$ otherwise $-e_j$. The universal property of CAR algebra 
ensures an automorphism $\alpha_-$ on $\cla_F$ via the second quantization $a(f) \raro a(Uf)$ i.e. 
$a(e_j) \raro a(e_j)$ if $j \ge 1$ otherwise $a(e_j) \raro -a(e_j)$ for $j < 1,$ 
where $e_j(i)=\delta^i_j$ is the Dirac function on $j$. Let 
$\hat{\cla_F}= \cla_F \Join  \IZ_2 $ be the cross-product $C^*$-algebra. 

\vsp 
Now we consider Pauli's $C^*$-algebra $\cla_P=\otimes_{k \in \IZ}\!M_2(\IC)$ with grading 
$\alpha_-(\sigma^k_z)=\sigma^k_z,\alpha(\sigma^k_x)=-\sigma^k_x$ and $\alpha(\sigma^k_y)
=-\sigma^k_y$, where $\sigma_x,\sigma_y,\sigma_z$ are Pauli spin matrices if $k < 1$ and 
$\alpha_-(\sigma^k_w)=\sigma^k_w$ for all $w=x,y,z$ for $k \ge 1$. We consider the 
cross-product $C^*$-algebra $\hat{\cla}_P=\cla_P \times| \IZ_2$. Jordan-Wigner transformation 
map $J$ which takes 
$$\sigma_z^i \raro 2a(e_j)^*a(e_j)-1,\;\sigma_x^j \raro TS_j(a(e_j)+a(e_j)^*), \sigma_y^j \raro TS_ji(a(e_j)-a(e_j)^*)$$
, where $T=\otimes_{k \le 0} \sigma_z^k$ and $S_j=1$ if $j=1$, 
$S_j=\otimes_{1 \le k \le j-1}\sigma_z^k$ if $j > 1$ and 
$S_j=\otimes_{j \le k \le 0}\sigma_z^k$ if $j < 1$ 
identifies $\hat{\cla}_F$ with $\hat{\cla}_P$ as the Jordan map is 
co-variant with grading automorphism $\alpha_-$ and thus via this 
map we have also identified $\cla^+_P$ with $\cla^+_F$. There is a 
one to one affine correspondence between the set of even states of 
$\cla_P$ and $\cla_F$.  

\vsp 
$XY$ model $H_{XY}= -\sum_k \sigma^k_x\sigma^{k+1}_x + \sigma^k_y\sigma^{k+1}_y$ 
is an even Hamiltonian and the KMS state $\omega_{\beta}$ at inverse temperature 
being unique $\omega_{\beta}$ is also an even state. Thus the low temperature limiting 
state is also even and the ground state being unique it is also pure and translation 
invariant. The unique ground state $\omega_{XY}$ of $XY$ model [AMa] once restricted to 
$\cla^+_P$ can be identified as restriction of a quasi-free state on $\cla_F$ to $\cla^+_F$. 

\vsp 
Now by a lemma of Antony Wassermann [Wa] ( page 496 ) we have the following: Since the closed 
real subspace $\clk=\{f \in \clh: f(x) = \overline{f(x)} \}$ which we identify with the closed subspace 
generated by $\{e_j:j \ge 1 \}$ satisfies the condition  $\clk + i \clk $ dense in $\clh$ ( we have equality ) and 
$\clk \bigcap i \clk = \{0 \}$, we get $\Omega$ is also cyclic and separating for $\pi(\cla^+_F)''$ i.e. 
$[\pi(\cla^+_F)''\zeta_{\omega}]= [\pi(\cla_F)''\zeta_{\omega}]$, where $(\clh,\pi,\Omega)$ is the GNS 
space of $(\cla_F,\omega_{XY})$ and $\omega_{XY}$ is the unique quasi-free state associated 
with unique ground state via Jordan map. Now going back to $\cla^+_P$, we find $\omega_{XY}$ on
$\cla^+_P$ is faithful.

\bigskip
{\centerline {\bf REFERENCES}}

\begin{itemize} 
\bigskip 

\item{[Ac1]} Accardi, L.: The non-commutative Markov property. (Russian) Funkcional. Anal. i Priložen. 9 (1975), no. 1, 1-8.

\item{[Ac2]} Accardi, L.: Non-commutative Markov chains associated to a preassigned evolution: an application to the quantum theory of measurement. Adv. in Math. 29 (1978), no. 2, 226-243. 

\item{[AC]} Accardi, Luigi; Cecchini, Carlo: Conditional expectations in von Neumann algebras and a theorem of Takesaki.
J. Funct. Anal. 45 (1982), no. 2, 245–273. 

\item {[AM]} Accardi, L., Mohari, A.: Time reflected Markov processes. Infin. Dimens. Anal. Quantum Probab. Relat. Top., vol-2, no-3, 
397-425 (1999).

\item{[AMa]} Araki, H., Matsui, T.: Ground states of the XY model, Commun. Math. Phys. 101, 213-245 (1985).

\item{[Br]} Bratteli, Ola,: Inductive limits of finite dimensional $C^*$-algebras.
Trans. Amer. Math. Soc. 171 (1972), 195-234.

\item{[BE]} Bratteli, Ola, Evans, David E.: Dynamical semigroups commuting with compact abelian actions. Ergodic Theory Dynam. Systems 3 (1983), no. 2, 187-217.

\item{[BR]} Bratteli, Ola., Robinson, D.W. : Operator algebras and quantum statistical mechanics, I,II, Springer 1981.

\item{[BJP]} Bratteli, Ola; Jorgensen, Palle E. T.; Price, Geoffrey L.: Endomorphisms of B(H). Quantization, nonlinear partial differential equations, 
and operator algebra (Cambridge, MA, 1994), 93–138, Proc. Sympos. Pure Math., 59, Amer. Math. Soc., Providence, RI, 1996. 

\item{[ChE]} Choi, Man-Duen; Effros, Edward G.: Nuclear $C^*$-Algebras and the Approximation Property, American Journal of Mathematics, Vol. 100, No.1, pp. 61-79 (1978).

\item{[CS]} Connes, A.; Størmer, E.: Entropy of automorphisms of II$_1$ -von Neumann
algebras, Acta Math. 134 (1975), 289-306.

\item{[CNT]} Connes, A., Narnhofer, H. and Thirring, W.: Dynamical entropy of $C^*$-
algebras and von Neumann algebras, Commun. Math. Phys. 112 (1987), 691-719.

\item{[CFS]} Cornfeld, I. P. Fomin, S. V.; Sinaĭ, Ya. G. Ergodic theory. Translated from the Russian by A. B. Sosinskiĭ. Grundlehren der Mathematischen Wissenschaften [Fundamental Principles of Mathematical Sciences], 245. Springer-Verlag, New York, 1982.

\item{[Cun]} Cuntz, J.: Simple $C\sp*$-algebras generated by isometries. Comm. Math. Phys. 57, 
no. 2, 173--185 (1977).  

\item{[DJ]} Dutkay, Dorin Ervin; Jorgensen, Palle E. T.: Monic representations of the Cuntz algebra and Markov measures. J. Funct. Anal. 267 (2014), no. 4, 1011-1034. 

\item{[ES]} Engelbert, H.J., Shiryaev, A.N.: On absolute continuity and singularity of probability measures, Mathematical Statistic, Banach center publications, vol-6 Warsaw, 1980. 

\item {[EvKa]} Evans, D. E.; Kawahigashi,Y.: Quantum symmetries on operator algebras, 
Oxford University Press. 

\item{[FNW1]} Fannes, M., Nachtergaele, B., Werner, R.: Finitely correlated states on quantum spin chains,
Commun. Math. Phys. 144, 443-490(1992).

\item{[FNW2]} Fannes, M., Nachtergaele, B., Werner, R.: Finitely correlated pure states, J. Funct. Anal. 120, 511-
534 (1994).

\item{[Gl]} Glimm, James G.: On a certain class of operator algebras. Trans. Amer. Math. Soc. 95, 318-340 (1960).

\item {[Ka]} Kadison, Richard V.: A generalized Schwarz inequality and algebraic invariants for operator algebras,  Ann. of Math. (2)  56, 494-503 (1952). 

\item {[Kak]} Kakutani, S.: On Equivalence of Infinite Product measures, Ann. of Math (2). 49, 214-226, (1948), 
 
\item{[La]} Lance, E. C.: Ergodic theorems for convex sets and operator algebras.
Invent. Math. 37 (1976), no. 3, 201-214. 

\item{[Ma]} Matsui, T.: Private communication (2011).

\item {[Mac]} Mackey, G.W.: Imprimitivity for representations of locally compact gropups I, Proc. Nat. Acad. Sci. U.S.A. 35 (1949), 537-545. 

\item {[Mo1]} Mohari, A.: Markov shift in non-commutative probability, Jour. Func. Anal. 199 (2003) 189-209.  

\item {[Mo2]} Mohari, A.: Pure inductive limit state and Kolmogorov property, J. Func. Anal. vol 253, no-2, 584-604 (2007)
Elsevier Sciences.

\item {[Mo3]} Mohari, A.: Jones index of a completely positive map, Acta Applicandae Mathematicae. Vol 108, Number 3, 665-677 

\item {[Mo4]} Mohari, A.: Pure inductive limit state and Kolmogorov property. II 
Journal of Operator Theory. vol 72, issue 2, 387-404.   
    
\item {[Mo5]} Mohari, A.: Translation invariant pure state on $\otimes_{k \in \IZ}\!M^{(k)}_d(\IC)$ and Haag duality, Complex Anal. Oper. Theory 8 (2014), no. 3, 745-789.

\item{[Mo6]} Mohari, A.: Translation invariant pure state on $\clb=\otimes_{k \in \IZ}\!M^{(k)}_d(\IC)$ and its split property, J. Math. Phys. 56, 061701 (2015).

\item {[Mu]}  Murphy, G. J.: C$^*$ algebras and Operator theory. Academic press, San Diego 1990. 

\item {[Na]} Nachtergaele, B. Working with quantum Markov states and their classical analogues. Quantum probability and applications, V (Heidelberg, 1988), 267-285,
Lecture Notes in Math., 1442, Springer, Berlin, 1990.

\item {[NeS]} Neshveyev, S.: St\o rmer, E. : Dynamical entropy in operator algebras. Springer-Verlag, Berlin, 2006. 

\item {[Or1]} Ornstein, D. S.: Two Bernoulli shifts with infinite entropy are isomorphic. Advances in Math. 5 1970 339-348 (1970).  

\item{[Or2]} Ornstein, D. S.: A K-automorphism with no square root and Pinsker's conjecture, 
Advances in Math. 10, 89-102. (1973).

\item {[OP]} Ohya, M., Petz, D.: Quantum entropy and its use, Text and monograph in physics, Springer-Verlag 1995.

\item {[Pau]} Paulsen, V.: Completely bounded maps and operator algebras, Cambridge Studies in Advance Mathematics 78, Cambridge University Press. 2002

\item {[Pa]} Parry, W.: Topics in Ergodic Theory, Cambridge University Press, 1981. 

\item {[Pow]} Powers, R. T.: Representation of uniformly hyper-finite algebras and their associated von-Neumann rings, Ann. Math. 86 (1967), 138-171. 

\item {[Sak]} Sakai, S.: C$^*$-algebras and W$^*$-algebras, Springer 1971.  

\item {[Si]} Sinai, Ja.: On the concept of entropy for a dynamic system. (Russian) Dokl. Akad. Nauk SSSR 124 1959 768-771. 
 
\item {[Sim]} Simon, B.: The statistical mechanics of lattice gases, vol-1, Princeton series in physics (1993). 

\item {[St]} Stinespring, W. F.: Positive functions on $C^*$ algebras, Proc. Amer. Math. Soc. 6 (1955) 211-216. 
 
\item {[SS]} Sinclair, Allan M.; Smith, Roger R.: Finite von Neumann algebras and masas. London Mathematical Society Lecture Note Series, 351. Cambridge University Press, Cambridge, 2008.

\item{[St\o 1]} St\o rmer, E. : Symmetric states of infinite tensor products of $C^*$- algebras. J. Functional Analysis, 3, 48-68 (1969)

\item{[St\o 2]} St\o rmer, E.: A survey of non-commutative dynamical entropy. Classification of nuclear $C^*$-algebras. Entropy in operator algebras, 147-198, Encyclopaedia Math. Sci., 126, Springer, Berlin, 2002.

\item{[Ta1]} Takesaki, M.: Conditional Expectations in von Neumann Algebras, J. Funct. Anal., 9, pp. 306-321 (1972)
 
\item{[Ta2]} Takesaki, M. : Theory of Operator algebras II, Springer, 2001.
  
\item {[Wa]} Wassermann, A.: Operator algebras and conformal field theory, Invent. Math. 133, 467-538 (1998).

\end{itemize}

\end{document}